\documentclass[a4paper,12pt]{article}
\usepackage[dvipdfmx]{graphicx}
\usepackage{amssymb}
\usepackage{amsmath}
\usepackage{amsthm}
\usepackage{verbatim,enumerate}
\usepackage[active]{srcltx}
\usepackage{comment}
\usepackage[noadjust]{cite}

 \newtheorem{theorem}{Theorem}[section]
 \newtheorem{proposition}[theorem]{Proposition}
 \newtheorem{fact}[theorem]{Fact}
 \newtheorem{lemma}[theorem]{Lemma}
 
 \newtheorem{introtheorem}{Theorem}
   
\theoremstyle{definition}
 \newtheorem{definition}[theorem]{Definition}

\numberwithin{equation}{section}
\numberwithin{figure}{section}

\usepackage[usenames]{color}

\makeatletter
	
	\@addtoreset{table}{section}
\makeatother

\title{Geometry of bifurcation sets of generic unfoldings
of corank two functions}
\author{Kentaro Saji \and Samuel P. Santos}
\date{\today}
\begin{document}
\maketitle
\renewcommand{\thefootnote}{\fnsymbol{footnote}}

\footnote[0]{
KS is partly supported by the
JSPS KAKENHI Grant numbered 18K03301
and the Japan-Brazil bilateral project JPJSBP1 20190103.
SS is partly supported by the grant 2019/10156-4, 
S\~ao Paulo Research Foundation (FAPESP).}
\newcommand{\toukouchange}[2]{\ifx\TOUKOU\undefined{#1}\else{#2}\fi}%
\newcommand{\vect}[1]{\boldsymbol{#1}}
\newcommand{\R}{\boldsymbol{R}}
\newcommand{\N}{\boldsymbol{N}}
\newcommand{\Z}{\boldsymbol{Z}}
\newcommand{\M}{\mathcal{M}}
\newcommand{\D}{\mathcal{D}}
\newcommand{\trace}{\operatorname{trace}}
\newcommand{\rank}{\operatorname{rank}}
\newcommand{\Hess}{\operatorname{Hess}}
\newcommand{\hess}{\operatorname{Hess}}
\newcommand{\image}{\operatorname{Im}}
\newcommand{\coef}{\operatorname{coef}}
\newcommand{\trans}[1]{{\vphantom{#1}}^t{\!#1}}
\renewcommand{\phi}{\varphi}
\newcommand{\sgn}{\operatorname{sgn}}
\newcommand{\inner}[2]{\left\langle{#1},{#2}\right\rangle}
\newcommand{\spann}[1]{\left\langle{#1}\right\rangle}
\newcommand{\Ker}{\operatorname{Ker}}
\newcommand{\corank}{\operatorname{corank}}
\newcommand{\ep}{\varepsilon}
\newcommand{\zv}{\vect{0}}
\newcommand{\x}{\vect{x}}
\renewcommand{\u}{\vect{u}}
\newcommand{\RR}{{\cal R}}
\newcommand{\B}{{\cal B}}
\newcommand{\C}{{\cal C}}
\newcommand{\pmt}[1]{{\begin{pmatrix} #1  \end{pmatrix}}}
\newcommand{\config}[4]{\{#1\,-\sqrt{3}\,#2\,0\,#3\,\sqrt{3}\,#4\}}
\allowdisplaybreaks[4]

\maketitle

\begin{abstract}
We study the geometry  of bifurcation sets of generic unfoldings
of $D_4^\pm$-functions.
Taking blow-ups, 
we show each of the bifurcation sets of $D_4^\pm$-functions
admit a parametrization as a surface in $\R^3$.
Using this parametrization, we investigate the behavior of
the Gaussian curvature and the principal curvatures. 
Furthermore, we investigate
the number of ridge curves and subparabolic curves
near their singular point.
\toukouchange{}{
}
\end{abstract}
\toukouchange{
}{}

\section{Introduction}
In recent decades, the differential geometry of fronts (wave fronts)
has been studied by many authors. In the Euclidean space,
the set consists of the collection of singular values
of a front and its parallel surfaces
is called the {\it caustic}.
Front and caustics are both fundamental objects in Lagrangian and
Legendrian singularity theory, and they are closely related (see 
\cite{AGV,izugra,irrt-book,zakall}, and also \cite{izutak,kruycaus}).
A front is a projection of the wave front set of an unfolding
of a function, and a caustic is the bifurcation set of an
unfolding of a function.
Although sometimes a singular point of a bifurcation set is
a singular point of a front,
the bifurcation sets of the versal unfolding of $D_4^\pm$-functions 
do not appear as fronts.
In this case, a parametrization of the bifurcation set
has not been given in the literature to the best knowledge of the authors.
In this paper, to see the geometry of the bifurcation set,
we simplify a versal unfolding $(\R^2\times\R^3,0)\to(\R,0)$ of a 
function $(\R^2,0)\to(\R,0)$
which is $\RR$-equivalent to the $D_4^\pm$-function
$$
f(u,v)=u^3/3!\pm uv^2/2,
$$
by using a coordinate change on $\R^2$ and 
an isometry on $\R^3$.
By using the blow-up method,
we give a parametrization 
of a generic versal unfolding of such a function,
and we show that
the parametrization 
is a front, and investigate 
its geometry.
For fronts, one can define classical differential
geometric invariants (Gaussian, mean and principal curvatures)
even though they diverge on the set of singular points.
We show:
\begin{introtheorem}\label{thm:bddunbdd}
One of the two principal curvatures of
the parametrization of the bifurcation set
of a versal unfolding of a\/ $D_4^\pm$-singularity\/
$C^\infty$-extends across the set of
singular points,
and the
other is unbounded near the singular point of the bifurcation set.
Moreover, the principal directions of these
principal curvatures\/
$C^\infty$-extend across the set of
singular points.
\end{introtheorem}
By using the asymptotic behavior of the Gaussian curvature,
we obtain the behavior of {\it parabolic curves\/} (the curves 
consisting of Gaussian curvature zero points) in
Theorems \ref{thm:numberpara} and \ref{thm:numberpara2}.
Moreover, by Theorem \ref{thm:bddunbdd}, we can discuss the conditions
for ridges and subparabolic curves near the singular point.
Let $g$ be an umbilic free regular surface,
and let $\kappa_i$ $(i=1,2)$ be the principal curvatures,
and $V_i$ the principal vector fields corresponding to $\kappa_i$.
A point $p$ on $g$ is
called a {\it ridge point with respect to\/ $\kappa_i$\/} if 
$V_i\kappa_i(p)=0$.
A point $p$ on $g$ is
called a {\it subparabolic point with respect to\/ $\kappa_i$\/} if 
$V_{j}\kappa_i(p)=0$, where $j=1,2$ and $j\ne i$.
A curve on $g$ is called a {\it ridge curve\/} (respectively,
{\it subparabolic curve\/}) if it consists of ridge points
(respectively, subparabolic points).
We show the following theorem
under the assumption that
the set of ridge points and the set of 
subparabolic points are curves.
See Section \ref{sec:numproof} and Proposition \ref{lem:ridges}
for the concrete conditions.
\begin{introtheorem}\label{lem:ridgenum}
We assume that the set of ridge points and the set of 
subparabolic points are curves.
Then the number of ridge curves with respect to the
unbounded principal curvature
emanating from the singular point
is at most\/ 
$18;$
the number of ridge curves with respect to the
bounded principal curvature
emanating from the singular point
is at most\/ 
$18;$
there are no subparabolic points with respect to the
unbounded principal curvature near
the singular point\/{\rm ;}
the number of subparabolic curves with respect to the
bounded principal curvature
emanating from the singular point
is at most\/ 
$10$.
\end{introtheorem}
\section{Preliminaries}
\subsection{Unfoldings and bifurcation sets}
Let $f:(\R^m,0)\to\R$ be a function.
A function $F:(\R^{m}\times\R^r,0)\to\R$ is called
an {\it unfolding\/} of $f$ if
$F(\u,0)=f(\u)$.
The {\it catastrophe set\/} $C_F$ of the unfolding $F$ of $f$ 
is
$$
C_F
=
\left\{(\u,\x)\in (\R^{m}\times\R^r,0)\,\left|\,
\dfrac{\partial F}{\partial u_1}(\u,\x)=\cdots=
\dfrac{\partial F}{\partial u_m}(\u,\x)=0\right.\right\},
$$
where $\u=(u_1,\ldots,u_m)$.
An unfolding $F:(\R^{m}\times\R^r,0)\to\R$ of 
$f:(\R^m,0)\to\R$ is a {\it Morse family of\/} $f$ if
$0\in \R^m$ is a critical point of $f$ and
\begin{equation}\label{eq:morse}
\rank\pmt{
\dfrac{\partial^2 F}{\partial u_1^2}&
\cdots&
\dfrac{\partial^2 F}{\partial u_1\partial u_m}&
\dfrac{\partial^2 F}{\partial u_1\partial x_1}&
\cdots&
\dfrac{\partial^2 F}{\partial u_1\partial x_n}\\
\vdots&\ddots&\vdots&\vdots&\ddots&\vdots\\
\dfrac{\partial^2 F}{\partial u_1\partial u_m}&
\cdots&
\dfrac{\partial^2 F}{\partial u_m^2}&
\dfrac{\partial^2 F}{\partial u_m\partial x_1}&
\cdots&
\dfrac{\partial^2 F}{\partial u_m\partial x_n}}
=m,
\end{equation}
at $0$ holds, where $\x=(x_1,\ldots,x_r)$.
By the implicit function theorem, if $F$ is a Morse family,
then $C_F$ is an $r$-dimensional submanifold of $(\R^{m}\times\R^r,0)$.
We set its parametrization
$B_1:C_F\to \R^{m}\times\R^3$ as an inclusion.
Let $\pi:\R^{m}\times\R^3\to\R^3$ be the projection,
and set $B_F=\pi\circ B_1$.
The singular set of $B_F$ is 
$
S(B_F)=\{(\u,\x)\in C_F\,|\,
\rank H_F(\u,\x)<m\}$, where
$$
H_F=\left(\dfrac{\partial^2 F}{\partial u_i\partial u_j}\right)
_{1\leq i,j\leq m}.
$$
The image $B_F(S(B_F))$ is called the {\it bifurcation set}, and
denoted by $\B_F$:
$$
\B_F=
\{\x\in\R^3\,|\,
\text{there exists }(\u,\x)\in C_F\text{ such that }
\rank H_F(\u,\x)<m\}.
$$
The following $P$-${\cal R}^+$ equivalence plays a fundamental role
in investigating bifurcation sets 
(see \cite[Chapter 8]{AGV},\cite[Chapter 5]{irrt-book} for details).
\begin{definition}
Let $F_i:(\R^{m}\times\R^r,0)\to \R$ be Morse families of 
functions $f_i:(\R^{m},0)\to \R$ $(i=1,2)$.
They are said to be {\it $P$-${\cal R}^+$ equivalent\/} if
there exists a triple $(g(\u,\x),G(\x),h(\x))$, where 
$g:(\R^m\times\R^r,0)\to (\R^m,0)$,
$G:(\R^r,0)\to(\R^r,0)$ is a diffeomorphism-germ,
and
$h:(\R^r,0)\to(\R,0)$ such that
\begin{equation}\label{eq:rplus}
F_2(\u,\x)=F_1(\bar G(\u,\x))+h(\x),\quad
(\bar G(\u,\x)=(g(\u,\x),G(\x))).
\end{equation}
\end{definition}
The following lemma is well-known 
(see \cite[ Proposition 3.1]{irrt-book} and its proof):
\begin{lemma}
Let\/ $F_i:(\R^{m}\times\R^r,0)\to \R$ be Morse families of\/
$f_i:(\R^{m},0)\to \R$ $(i=1,2)$.
If\/ $F_1,F_2$ are\/ $P$-$\RR^+$-equivalent as in\/
\eqref{eq:rplus},
then\/ 
$G(\B_{F_1})=\B_{F_2}$ as set germs at\/ $0$.
\end{lemma}
By this lemma, to investigate the geometry of bifurcation sets, 
with respect to Euclidean geometry in $\R^3$, 
we introduce the following $P$-$\RR^+$-isometricity:
\begin{definition}
Let $F_i:(\R^{m}\times\R^r,0)\to \R$ be Morse families of 
functions $f_i:(\R^{m},0)\to \R$ $(i=1,2)$.
They are said to be {\it $P$-${\cal R}^+$-isometric\/} if
they are $P$-${\cal R}^+$-equivalent, and the diffeomorphism-germ
$G:(\R^r,0)\to(\R^r,0)$ in the triple 
$(g(\u,\x),$ $G(\x),$ $h(\x))$ which gives $P$-${\cal R}^+$-equivalence
is an isometry-germ.
\end{definition}

We may simplify the functions $f$ and $F$ by $P$-$\RR^+$-isometry.
Two functions $f_i:(\R^{m},0)\to \R$ $(i=1,2)$
are {\it $\RR$-equivalent\/} if there exists a diffeomorphism-germ
$\phi:(\R^{m},0)\to(\R^{m},0)$ such that $f_1=f_2\circ \phi$.
A function-germ $f$ at $0$ of two variables is a
{\it $D_4^\pm$-germ\/}
if it is $\RR$-equivalent to $f(u,v)=u^3/3!\pm uv^2/2$,
where $(u_1,u_2)$ is denoted by $(u,v)$.
Let $\phi:(\R^{m},0)\to(\R^{m},0)$ be a diffeomorphism-germ.
Then 
a Morse family $F:(\R^{m}\times\R^r,0)\to \R$  of 
$f:(\R^{m},0)\to \R$ is $P$-$\RR^+$-equivalent to
the unfolding $F(\phi(\u),\x)$.
Since we study the geometry of the bifurcation set of
a $D_4^\pm$-germ $f$, under the $P$-${\cal R}^+$-isometricity,
we may assume
$f=u^3/3!\pm uv^2/2$ without loss of generality.

\subsection{Simplification of an unfolding by $P$-$\RR^+$-equivalence}
\begin{definition}
Let $F_i:(\R^{m}\times\R^{r_i},0)\to \R$ $(i=1,2)$
be two unfoldings of $f:(\R^m,0)\to\R$.
An {\it $\RR^+$-$f$-morphism from\/ $F_2$ to\/ $F_1$} is a triple 
$(g(\u,\x),$ $G(\x),$ $h(\x))$, where 
$g:(\R^m\times\R^{r_2},0)\to (\R^m,0)$,
$G:(\R^{r_2},0)\to(\R^{r_1},0)$,
$h:(\R^{r_2},0)\to(\R,0)$,
and they satisfy $g(\u,0)=\u$ and
\begin{equation}\label{eq:rfmorph}
F_2(\u,\x)=F_1(g(\u,\x),G(\x))+h(\x).
\end{equation}
\end{definition}
\begin{definition}
An unfolding $F_1:(\R^{m}\times\R^{r_1},0)\to \R$ of $f:(\R^m,0)\to\R$
is an $\RR^+$-{\it versal unfolding\/} if for any 
unfolding $F_2:(\R^{m}\times\R^{r_2},0)\to \R$ of $f$,
there exists an $\RR^+$-$f$-morphism from $F_2$ to $F_1$.
\end{definition}

It is known that the function
$F_{0,\ep_1}(u,v,x,y,z)=u^3/3!+\ep_1 uv^2/2+xu+yv+z(u^2-\ep_1v^2)/2$ 
is an\/ $\RR^+$-versal unfolding 
of\/ $u^3/3!+\ep_1 uv^2/2$, where\/ $\ep_1=\pm1$
(\cite[Chapter 8]{AGV},\cite[Chapter 5]{irrt-book}).
We call 
the bifurcation set of an $\RR^+$-versal unfolding of $u^3/3!\pm uv^2/2$
a {\it $D_4^\pm$-singularity}.
The bifurcation set of 
$F_{0,\ep_1}$ is the set
$$
B_{0,\ep_1}=\{(-u^2/2-\ep_1v^2/2-zu,-\ep_1uv+\ep_1zv,z)\,|\,
\ep_1(u^2-z^2)-v^2=0\}. 
$$
We can observe that $B_{0,1}$ consists of two sheets and
they have intersection curves
$c(t)=\{(-t,\pm t,0)\,|\,t>0\}$.
All $D_4^+$-singularities are locally diffeomorphic to
$B_{0,1}$, we call the corresponding curves to $c(t)$ the
{\it intersection curve}.
See Figure \ref{fig:biffig} for 
the bifurcation sets of $F_{0,\ep_1}$.
\begin{figure}[ht]
\centering
\includegraphics[width=.3\linewidth]{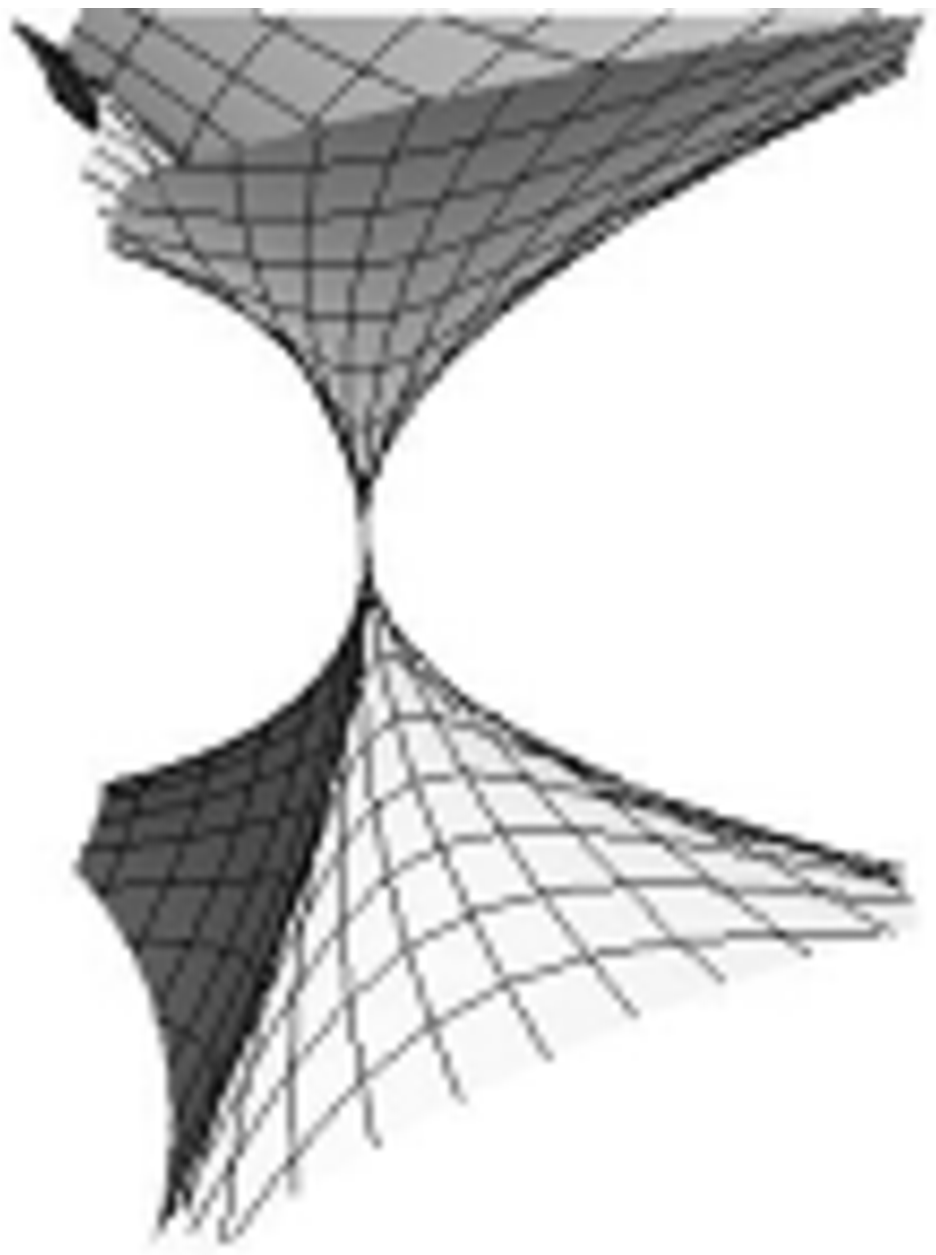}
\hspace{10mm}
\includegraphics[width=.3\linewidth]{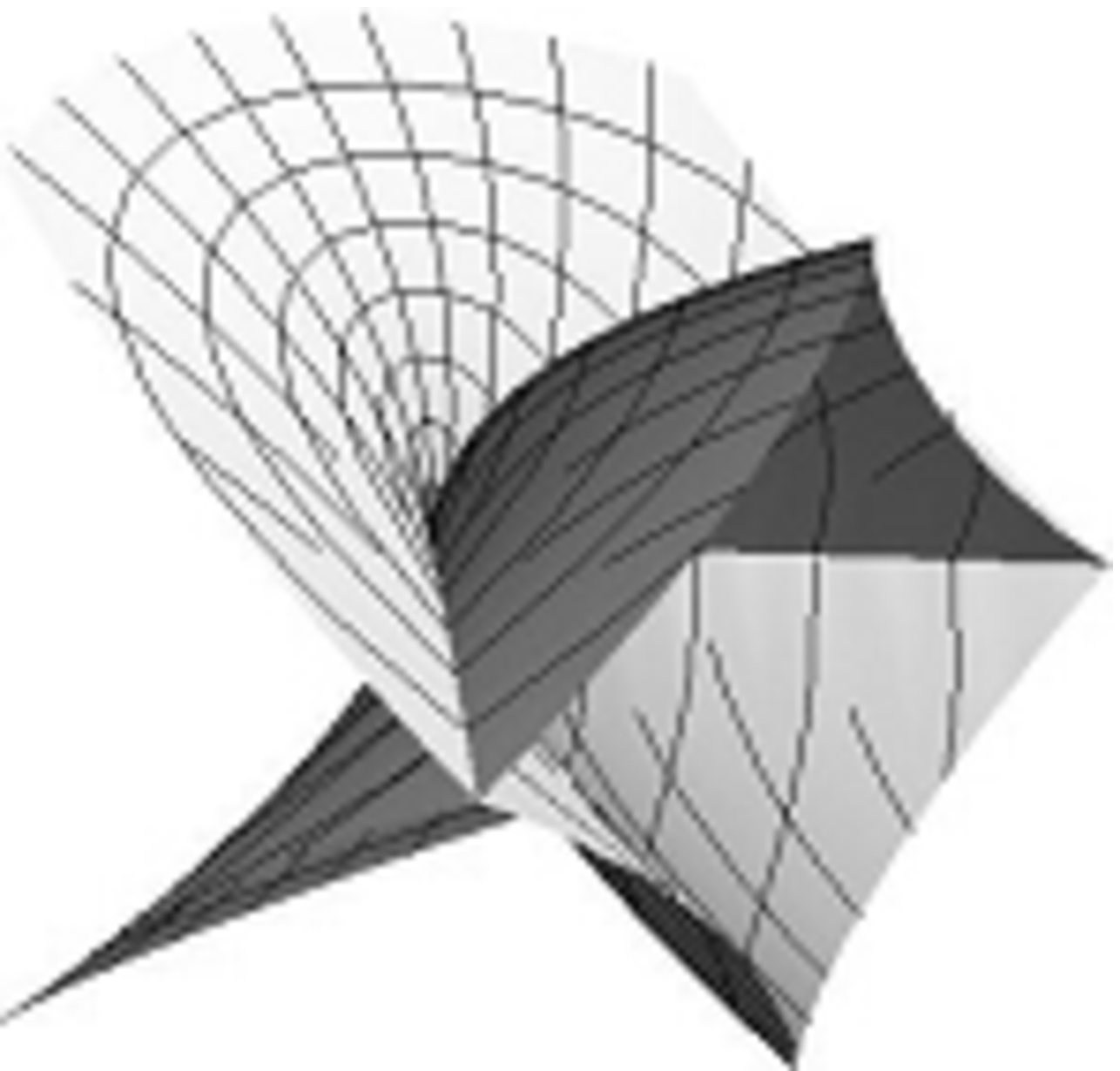}
\caption{The bifurcation sets of $F_{0,-1}$ and
$F_{0,1}$}
\label{fig:biffig}
\end{figure}

We remark that $\RR^+$-versal unfolding $(\R^2\times\R^3,0)\to(\R,0)$ 
of the $D_4^\pm$-germ
is unique (\cite[Chapter 8]{AGV},\cite[Chapter 5]{irrt-book}).

Let $f:(\R^2,0)\to\R$ be
$f(u,v)=u^3/3!+\ep_1 uv^2/2$ $(\ep_1=\pm1)$
and let $F:(\R^2\times\R^3,0)\to\R$ be an $\RR^+$-versal unfolding of $f$.
We have the following:
\begin{proposition}
The function $F$ is $P$-$\RR^+$-isometric to
\begin{equation}\label{eq:normal1}
F(\u,\x)=F_0(\u,G(\x))=
\dfrac{u^3}{3!}+\ep_1\dfrac{uv^2}{2}
+P(\x)u+Q(\x)v+R(\x)\dfrac{u^2-\ep_1v^2}{2}
\end{equation}
with the condition
\begin{equation}\label{eq:normal2}
g_{1,010}=g_{1,001}=g_{2,001}=0
\quad\text{and}\quad
g_{1,100},\ g_{2,010},\ g_{3,001}>0,
\end{equation}
where
\begin{equation}\label{eq:normal3}
G_n(\x)=\sum_{i+j+k\geq 1}\dfrac{g_{n,ijk}}{i!j!k!}x_1^ix_2^jx_3^k
\quad(G_1=P,\ G_2=Q,\ G_3=R,\ n=1,2,3).
\end{equation}
\end{proposition}
\begin{proof}
Since $F_{0,\ep_1}$ is a versal unfolding,
there exist 
$g:(\R^2\times\R^3,0)\to (\R^2,0)$,
$G:(\R^3,0)\to(\R^3,0)$,
$h:(\R^3,0)\to(\R,0)$,
and they satisfy $g(\u,0)=\u$ and
$F(\u,\x)=F_0(g(\u,\x),G(\x))+h(\x)$.
By the triple $(g(\u,\x),\x,h(\x))$, 
we see $F$ is $P$-$\RR^+$-isometric to
$F_0(u,G(\x))$ as in \eqref{eq:normal1}.
Let us set
$$
\vect{g}_1=\pmt{g_{1,100}\\g_{1,010}\\g_{1,001}},\quad
\vect{g}_2=\pmt{g_{2,100}\\g_{2,010}\\g_{2,001}},\quad
\vect{g}_3=\pmt{g_{3,100}\\g_{3,010}\\g_{3,001}},
$$
and let $\overline{\vect{g}_1},\overline{\vect{g}_2},\overline{\vect{g}_3}$
be the Gram-Schmidt orthonormalized vectors, namely,
\begin{align*}
\overline{\vect{g}_1}
=&\dfrac{\vect{g}_1}{|\vect{g}_1|},\qquad
\widetilde{\vect{g}_2}
=\vect{g}_2-(\overline{\vect{g}_1}\cdot\vect{g}_2)\overline{\vect{g}_1},\qquad
\overline{\vect{g}_2}
=\dfrac{\widetilde{\vect{g}_2}}{|\widetilde{\vect{g}_2}|},\\
\widetilde{\vect{g}_3}
=&\vect{g}_3-
\Big((\overline{\vect{g}_1}\cdot\vect{g}_3)\overline{\vect{g}_1}
+(\overline{\vect{g}_2}\cdot\vect{g}_3)\overline{\vect{g}_2}\Big),\qquad
\overline{\vect{g}_3}
=\dfrac{\widetilde{\vect{g}_3}}{|\widetilde{\vect{g}_3}|}.
\end{align*}
We set 
$$
M=\pmt{
\trans{\overline{\vect{g}_1}}\\
\trans{\overline{\vect{g}_2}}\\
\trans{\overline{\vect{g}_3}}},
$$
where $\trans{(~)}$ stands for matrix transposition.
Then $M$ is an orthonormal matrix, and is identified with a
linear map.
By the triple $(\operatorname{id},M,0)$, we see the first
condition of
\eqref{eq:normal2} can be satisfied.
By the versality, 
$g_{1,100}g_{2,010}g_{3,001}\ne0$.
The second condition of \eqref{eq:normal2} can be satisfied 
by the linear map defined by the orthonormal matrix
$$
\pmt{
\pm1&0&0\\
0&\pm1&0\\
0&0&\pm1}.
$$
\end{proof}
Geometric meanings of coefficients of $G_1,G_2,G_3$ 
are discussed in Section \ref{sec:geombif}.

\subsection{Fronts}
Since we shall show the bifurcation sets are fronts,
we give a fundamental definition of fronts.
Let $f:(\R^2,0)\to(\R^3,0)$ be a map-germ.
The map $f$ is a {\it frontal\/} if there exists a
unit normal vector field $\nu$ along $f$ such that
$\inner{df_p(X)}{\nu(p)}=0$ for any $X\in T_p(\R^2,0)$.
A frontal $f$ with a unit normal vector $\nu$ is a 
{\it front\/} if the pair $(f,\nu)$ is an immersion.
Let $f$ be a frontal. 
A function $\lambda:(\R^2,0)\to\R$ is called an {\it identifier of singularities}
if it is a non-zero multiple of the function $\det(f_u,f_v,\nu)$.
If a function is an identifier of singularities, then
the set of singular points $S(f)$ satisfies $S(f)=\lambda^{-1}(0)$.
Since the unit normal vector field is well-defined for
frontals,
one can define the Gaussian, mean and principal curvatures
in a natural way. However, it may diverge on the set of singular points.
See \cite{ms,msuy,osetari,suyfro,teraprin} for 
differential geometric study of these curvatures.
Let $f:(\R^2,0)\to(\R^3,0)$ be a front with a unit normal vector field $\nu$,
which is a cuspidal edge
(by coordinate transformations on the source and the target space,
it can be written as $(u,v)\mapsto(u,v^2,v^3)$).
Let $\gamma:(J,0)\to (\R^3,0)$ be a parametrization of the set of singular points
of $f$, where $J$ is an open interval containing $0$.
We set $\hat\gamma(t)=f\circ \gamma(t)$ and $\hat\nu(t)=\nu\circ\gamma(t)$.
The {\it singular curvature\/} $\kappa_s$ 
and the ({\it limiting\/}) {\it normal curvature\/} $\kappa_n$
are defined by 
\begin{equation}\label{eq:kskn}
\kappa_s(t)=\pm\dfrac{\det(\hat\gamma',\hat\gamma'',\hat\nu)}{|\hat\gamma'|^3}(t),\quad
\kappa_n(t)=\dfrac{\hat\gamma''\cdot\hat\nu}{|\hat\gamma'|^2}(t).
\end{equation}
See \cite{suyfro} for details.
\section{Description of bifurcation sets}
We assume $f(u,v)=u^3/3!+\ep_1 uv^2/2$ and $F$ is
written as \eqref{eq:normal1} with the conditions
\eqref{eq:normal2}.
Then by the implicit function theorem,
there exist two functions $x(u,v,z),y(u,v,z)$ such that
$$F_u(u,v,x(u,v,z),y(u,v,z),z)=F_v(u,v,x(u,v,z),y(u,v,z),z)=0.$$
Thus $C_F$ can be parametrized by 
$(u,v,z)\mapsto (u,v,x(u,v,z),y(u,v,z),z)$,
and the map $B_F$ is $B_F(u,v,z)=(x(u,v,z),y(u,v,z),z)$.
Since
\begin{equation}\label{eq:dbf}
dB_F=\pmt{
x_u&x_v&x_z\\
y_u&y_v&y_z\\
0&0&1\\
},
\end{equation}
it holds that $S(B_F)=\{x_uy_v-x_vy_u=0\}$.
By the implicit function theorem,
\begin{equation}\label{eq:impld4}
\pmt{
F_{ux}&F_{uy}\\
F_{vx}&F_{vy}}
\pmt{
x_u&x_v\\
y_u&y_v}
=-
\pmt{
F_{uu}&F_{uv}\\
F_{uv}&F_{vv}}
\end{equation}
and by \eqref{eq:normal2},
the bifurcation set can be written as
\begin{equation}\label{eq:bifd4}
\B_F=\{(x,y,z)\,|\,\text{ there exists }(u,v)\in C_F\text{ such that }
\det H_F(u,v,x,y,z)=0\}.
\end{equation}
\subsection{$D_4^-$ singularity}\label{sec:d4mdisc}
We give a parametrization of the bifurcation set
by using the blow-up method at a singular point \cite{fukuihase}
(see also \cite[Example (a) in p. 221]{hironaka}).
Let $S^1=\R/2\pi\Z$ be a circle, and let $I=(-\ep,\ep)$ be
an open interval. Two points $(\theta_i,r_i)\in S^1\times I$ 
$(i=1,2)$ are equivalent $(\sim)$ if $(\theta_2,r_2)=(\theta_1+\pi,r_1)$.
The quotient space $\M=S^1\times I/\sim$ is topologically a M\"obius strip.
There is a natural map $\pi:\M\to \R^2$, 
where $\pi([(\theta,r)])=(r\cos\theta,r\sin\theta)$.
This is usually called a blow-up.
Furthermore, we take a double cover $\hat\M$ of $\M$,
and consider a natural map $\hat\pi:\hat\M\to \R^2$, 
where $\hat\pi([(\theta,r)])=(r\cos2\theta,r\sin2\theta)$.
Then $\hat\M$ is topologically an annulus.
We assume $\ep_1=-1$ in \eqref{eq:normal1}.
Then
$\det H_F=-u^2-v^2+R(x,y,z)^2$.
We set
\begin{equation}\label{eq:impld4040}
u=r\cos2\theta,\quad
v=r\sin2\theta,
\end{equation}
where $(\theta,r)\in \hat\M$.
Then $-r^2+R^2=0$ can be solved by
$r=\ep_2 R(x,y,z)$, where $\ep_2=\pm1$.
Then the equation for the bifurcation set is
\begin{equation}\label{eq:impld4050}
F_u(2\theta,r,x,y,z)=0,\ 
F_v(2\theta,r,x,y,z)=0,\ 
r=\ep_2 R(x,y,z).
\end{equation}
This is equivalent to
\begin{equation}\label{eq:impld4055}
X_{-1,\ep_2}(\theta,x,y,z)=0,\ Y_{-1,\ep_2}(\theta,x,y,z)=0,
\end{equation}
where 
\begin{equation}\label{eq:xyalphabeta1}
\begin{array}{rl}
X_{-1,\ep_2}(\theta,x,y,z)&=F_u(\ep_2 R(x,y,z),2\theta,x,y,z)\\
&=
P(x,y,z)
+R(x,y,z)^2\alpha_{-1,\ep_2}(\theta),\\
Y_{-1,\ep_2}(\theta,x,y,z)&=F_v(\ep_2 R(x,y,z),2\theta,x,y,z)\\
&=
Q(x,y,z)
+R(x,y,z)^2\beta_{-1,\ep_2}(\theta),\\
\alpha_{-1,\ep_2}(\theta)&=
\dfrac{\cos ^22\theta-\sin^22\theta }{2}
+\ep_2 \cos 2\theta,\\
\beta_{-1,\ep_2}(\theta)
&=
-\cos 2\theta \sin2\theta 
+\ep_2 \sin2\theta.
\end{array}
\end{equation}
By the condition \eqref{eq:normal2}, 
there exist functions $x_{-1,\ep_2}(\theta,z)$, $y_{-1,\ep_2}(\theta,z)$
such that 
$X_{-1,\ep_2}$ $(\theta,x_{-1,\ep_2}$ $(\theta,z),$ 
$y_{-1,\ep_2}(\theta,z),z)=
Y_{-1,\ep_2}$ $(\theta,x_{-1,\ep_2}(\theta,z),y_{-1,\ep_2}(\theta,z),z)$
$=0$
holds identically.
In this setting, we have:
\begin{lemma}\label{lem:xyg3}It holds that
\begin{equation}\label{eq:impld4535}
x_{-1,\ep_2}(\theta,0)=y_{-1,\ep_2}(\theta,0)=
(x_{-1,\ep_2})_z(\theta,0)=(y_{-1,\ep_2})_z(\theta,0)=0
\end{equation}
and
\begin{equation}\label{eq:impld4536}
R(x_{-1,\ep_2}(\theta,0),y_{-1,\ep_2}(\theta,0),0)=0.
\end{equation}
\end{lemma}
\begin{proof}
We consider equations $X_{-1,\ep_2}(\theta,x,y,0)
=Y_{-1,\ep_2}(\theta,x,y,0)=0$.
Then since \eqref{eq:normal2}, we have
functions $\bar x_{-1,\ep_2}(\theta),\bar y_{-1,\ep_2}(\theta)$ such that
\begin{equation*}
X_{-1,\ep_2}(\theta,\bar x_{-1,\ep_2}(\theta),\bar y_{-1,\ep_2}(\theta),0)
=Y_{-1,\ep_2}(\theta,\bar x_{-1,\ep_2}(\theta),\bar y_{-1,\ep_2}(\theta),0)=0.
\end{equation*}
We remark that by the implicit function theorem, 
each $\bar x_{-1,\ep_2}(\theta),\bar y_{-1,\ep_2}(\theta)$ is unique.
On the other hand, 
since $P(0,0,0)=Q(0,0,0)=0$, the equality
$X_{-1,\ep_2}(\theta,0,0,0)=Y_{-1,\ep_2}(\theta,0,0,0)=0$ is satisfied.
Thus $\bar x_{-1,\ep_2}(\theta)=\bar y_{-1,\ep_2}(\theta)=0$ is a solution of
$X_{-1,\ep_2}(\theta,x,y,0)=Y_{-1,\ep_2}(\theta,x,y,0)=0$.
By the uniqueness, $\bar x_{-1,\ep_2}(\theta)=\bar y_{-1,\ep_2}(\theta)=0$ holds.
By the definition of the functions
$\bar x_{-1,\ep_2}(\theta), \bar y_{-1,\ep_2}(\theta)$, it holds that
$\bar x_{-1,\ep_2}(\theta)=x_{-1,\ep_2}(\theta,0), \bar y_{-1,\ep_2}(\theta)=y_{-1,\ep_2}(\theta,0)$.
This shows $x_{-1,\ep_2}(\theta,0)=0$ and $y_{-1,\ep_2}(\theta,0)=0$.
Moreover, by \eqref{eq:xyalphabeta1} and \eqref{eq:normal2}, 
it holds that $(x_{-1,\ep_2})_z(\theta,0)=(y_{-1,\ep_2})_z(\theta,0)=0$.
The equation \eqref{eq:impld4536}
is obvious by \eqref{eq:impld4535}.
\end{proof}
Thus a double cover of the bifurcation set can be parameterized by
\begin{equation}\label{eq:impld4070}
b_{-1,\ep_2}(\theta,z)=b(\theta,z)
=(x_{-1,\ep_2}(\theta,z),y_{-1,\ep_2}(\theta,z),z):\hat{\M}\to\R^3
\quad(\ep_2=\pm1).
\end{equation}
We define the source space $\hat{\M}$ of $b_{-1,\ep_2}$ by 
$\hat{\M}_{\ep_2}$. 
On the other hand, since $R_z(0,0,0)=g_{3,001}\ne0$, and 
$$\alpha_{-1,\ep_2}(\theta)=\alpha_{-1,\ep_2}(\theta+\pi/2),\ 
\beta_{-1,\ep_2}(\theta)=\beta_{-1,\ep_2}(\theta+\pi/2),$$
we can regard $(\theta,z)\in \M$ as a parameter of the
bifurcation set.
We have the following proposition.
\begin{proposition}
The map\/ $b_{-1,\ep_2}:\hat\M_{\ep_2}\to\R^3$ is 
a front near\/ $\{z=0\}\subset \hat\M$.
\end{proposition}
\begin{proof}
By \eqref{eq:impld4536}, there exists a function 
$\tilde R_{-1,\ep_2}(\theta,z)$
such that
$R(\theta,x_{-1,\ep_2}(\theta,z),$ $y_{-1,\ep_2}(\theta,z),z)$
$=z\tilde R_{-1,\ep_2}(\theta,z)$.
By \eqref{eq:impld4535} and $R_z(0,0,0)\ne0$, it holds that
$\tilde R_{-1,\ep_2}(\theta,0)\ne0$.
By a direct calculation,
\begin{equation}\label{eq:impld4130}
\begin{array}{ll}
&\Big((X_{-1,\ep_2})_\theta(x_{-1,\ep_2}(\theta,z),y_{-1,\ep_2}(\theta,z),z)
,\\
&\hspace{10mm}
(Y_{-1,\ep_2})_\theta(x_{-1,\ep_2}(\theta,z),y_{-1,\ep_2}(\theta,z),z),0\Big)
\\
&\hspace{20mm}=
\left\{
\begin{array}{ll}
4 z^2\tilde R_{-1,\ep_2}(\theta,z)^2
\sin3\theta\,\trans{(-\cos\theta, \sin\theta,0)}
\ (\ep_2=1),\\
-4 z^2\tilde R_{-1,\ep_2}(\theta,z)^2
\cos3\theta\,\trans{(\sin\theta, \cos\theta,0)}
\ (\ep_2=-1)
\end{array}\right.
\end{array}\end{equation}
holds. 
By the implicit function theorem,
\begin{align}
\label{eq:impld4150}
\pmt{(x_{-1,\ep_2})_\theta\\ (y_{-1,\ep_2})_\theta}
=&
-z^2\lambda_{-1,\ep_2}a_{-1,\ep_2}
\tilde A_{-1,\ep_2}V _{-1,\ep_2},\\
\tilde A_{-1,\ep_2}
=&
\pmt{(Y_{-1,\ep_2})_y&-(X_{-1,\ep_2})_y\\-(Y_{-1,\ep_2})_x&(X_{-1,\ep_2})_x},
\nonumber\\
\pmt{(x_{-1,\ep_2})_z\\ (y_{-1,\ep_2})_z}=&
\dfrac{-1}{\det\tilde A_{-1,\ep_2}}
\tilde A_{-1,\ep_2}
\pmt{(X_{-1,\ep_2})_z \\ (Y_{-1,\ep_2})_z},\label{eq:impld4160}
\\
V _{-1,\ep_2}=&\left\{
\begin{array}{ll}
\trans{(-\cos\theta,\sin\theta)}\ &(\ep_2=1)\\
\trans{( \sin\theta,\cos\theta)}\ &(\ep_2=-1),
\end{array}\right.\label{eq:impld4170}\\
\lambda_{-1,\ep_2}=&
\left\{
\begin{array}{ll}
\sin3\theta\ &(\ep_2=1)\\
\cos3\theta\ &(\ep_2=-1),
\end{array}
\right.\label{eq:impld4180}\\
a_{-1,\ep_2}=&\dfrac{4\tilde R_{-1,\ep_2}(\theta,z)^2}
{\det\tilde A_{-1,\ep_2}}\label{eq:impld4190}
\end{align}
hold. Thus
$b_\theta\times b_z$
is proportional to
\begin{align}
\label{eq:impld4630}
&\tilde\nu_{-1,\ep_2}(\theta,z)\\
=&
\left\{
\begin{array}{l}
\big(\sin\theta dX_{-1,\ep_2}+\cos\theta dY_{-1,\ep_2}\big)
(\theta,x_{-1,\ep_2}(\theta,z),y_{-1,\ep_2}(\theta,z),z)
\ (\ep_2=1),\\[1mm]
\big(\cos\theta dX_{-1,\ep_2}-\sin\theta dY_{-1,\ep_2}\big)
(\theta,x_{-1,\ep_2}(\theta,z),y_{-1,\ep_2}(\theta,z),z)
\ (\ep_2=-1),
\end{array}\right.\nonumber
\end{align}
where 
\begin{align*}
dX_{-1,\ep_2}
&=
\trans{\big((X_{-1,\ep_2})_x,(X_{-1,\ep_2})_y,(X_{-1,\ep_2})_z\big)},\\
dY_{-1,\ep_2}
&=
\trans{\big((Y_{-1,\ep_2})_x,(Y_{-1,\ep_2})_y,(Y_{-1,\ep_2})_z\big)}.
\end{align*}
Since 
$\tilde\nu_{-1,\ep_2}(\theta,0)\ne0$, 
the unit vector
$\nu_{-1,\ep_2}=\tilde\nu_{-1,\ep_2}/|\tilde\nu_{-1,\ep_2}|$ 
is a unit normal vector of $b$.
Since $b_\theta(\theta,0)=0$, to show that $b$ is a front
it is enough to see
$\tilde\nu_{-1,\ep_2}(\theta,0)$ and 
$(\tilde\nu_{-1,\ep_2})_\theta(\theta,0)$ are linearly independent.
It is easy to see by $dX$ and $dY$ are linearly independent
at $(\theta,0)$. This shows the assertion.
\end{proof}
The singular set $S(b)$ is
\begin{align}
\label{eq:impld4195}
&S(b)\\
=&
\left\{
\begin{array}{ll}
\{(\theta,z)\,|\,\sin3\theta=0\}=
\{(n\pi/3,z)\,|\,n=0,1,2,3,4,5\}
\ &(\ep_2=1),\\
\{(\theta,z)\,|\,\cos3\theta=0\}
=
\{(n\pi/3+\pi/2,z)\,|\,n=-1,0,1,2,3,4\}
\ &(\ep_2=-1).
\end{array}\right.
\nonumber
\end{align}
\subsection{$D_4^+$ singularity}\label{sec:d4pdisc}
We will give a parametrization of the bifurcation set
by a similar method to that of Section \ref{sec:d4mdisc}
in the case of $\ep_1=1$.
Let $I=(-\ep,\ep)$ be an open interval. 
We consider a map $\pi:\R\times I\to \R^2$ defined by 
$\pi_1((\theta,r))=(r\cosh\theta,r\sinh\theta)$.

We assume $\ep_1=1$ in \eqref{eq:normal1}.
Then
$\det H_F=u^2-v^2+R(x,y,z)^2$.
We set
\begin{equation}\label{eq:impld4520}
u=r\cosh2\theta,\quad
v=r\sinh2\theta.
\end{equation}
Then $r^2+R^2=0$ can be solved by 
$r=\ep_2 R$, where $\ep_2=\pm1$.
We set
\begin{equation}\label{eq:impld4540}
\begin{array}{rl}
X_{1,\ep_2}=&
P(x,y,z)+R(x,y,z)^2\alpha_{1,\ep_2}(\theta),\\
Y_{1,\ep_2}=&
Q(x,y,z)+R(x,y,z)^2\beta_{1,\ep_2}(\theta),\\
\alpha_{1,\ep_2}(\theta)=&\dfrac{\cosh 2\theta^2+\sinh2\theta^2}{2}
+\ep_2 \cosh 2\theta,\\
\beta_{1,\ep_2}(\theta)=&\cosh 2\theta \sinh2\theta
-\ep_2 \sinh2\theta.
\end{array}
\end{equation}
There exist $x_{1,\ep_2}(\theta,z)$ and $y=y_{1,\ep_2}(\theta,z)$
such that 
$$X_{1,\ep_2}(\theta,x_{1,\ep_2}(\theta,z),y_{1,\ep_2}(\theta,z),z)=
Y_{1,\ep_2}(\theta,x_{1,\ep_2}(\theta,z),y_{1,\ep_2}(\theta,z),z)=0$$ 
holds identically.
By the same proof as for Lemma \ref{lem:xyg3},
we have 
$x_{1,\ep_2}(\theta,0)=y_{1,\ep_2}(\theta,0)=0$,
$(x_{1,\ep_2})_z(\theta,0)=(y_{1,\ep_2})_z(\theta,0)=0$,
$R(x_{1,\ep_2}(\theta,0),y_{1,\ep_2}(\theta,0),0)=0$.
Thus 
there exists a function $\tilde R_{1,\ep_2}(\theta,z)$
such that
$$R(\theta,x_{1,\ep_2}(\theta,z),y_{1,\ep_2}(\theta,z),z)=
z\tilde R_{1,\ep_2}(\theta,z).$$
The bifurcation set, except for its intersection set, 
can be parameterized by
$b(\theta,z)=b_{1,\ep_2}(\theta,z)=(x_{1,\ep_2}(\theta,z),y_{1,\ep_2}(\theta,z),z)$.
We define the source space $\R\times I$ of $b_{1,\ep_2}$ by 
${\cal D}_{\ep_2}$. 
By a similar argument as in the case of a $D_4^-$ singularity,
setting
\begin{align}
V_{1,\ep_2} =&
\left\{
\begin{array}{ll}
\trans{(\cosh\theta, \sinh\theta)}\ &(\ep_2=1)\\
\trans{(\sinh\theta, \cosh\theta)}\ &(\ep_2=-1),
\end{array}
\right.\label{eq:bifhyp0100}\\
\lambda_{1,\ep_2}=&
\left\{
\begin{array}{ll}
\sinh3\theta\ &(\ep_2=1)\\
\cosh3\theta\ &(\ep_2=-1),
\end{array}
\right.\label{eq:bifhyp0200}\\
a_{1,\ep_2}=&
\dfrac{4\tilde R_{1,\ep_2}^2}{\det \tilde A_{1,\ep_2}},\quad
\tilde A_{1,\ep_2}=
\pmt{(Y_{1,\ep_2})_y&-(X_{1,\ep_2})_y\\-(Y_{1,\ep_2})_x&(X_{1,\ep_2})_x},
\label{eq:bifhyp0300}
\end{align}
we have
\begin{equation}\label{eq:impld4620}
((x_{1,\ep_2})_\theta,(y_{1,\ep_2})_\theta)
=-
z^2\lambda_{1,\ep_2}a_{1,\ep_2}\tilde A_{1,\ep_2}V _{1,\ep_2}.
\end{equation}
We set
\begin{align}
\label{eq:impld4635}
&\tilde\nu_{1,\ep_2}(\theta,z)\\
=&\left\{
\begin{array}{ll}
(-\cosh\theta dY_{1,\ep_2}+\sinh\theta dX_{1,\ep_2})
(x_{1,\ep_2}(\theta,z),y_{1,\ep_2}(\theta,z),z)\ &(\ep_2=1)\\[1mm]
(-\sinh\theta dY_{1,\ep_2}+\cosh\theta dX_{1,\ep_2})
(x_{1,\ep_2}(\theta,z),y_{1,\ep_2}(\theta,z),z)\ &(\ep_2=-1).
\end{array}\right.
\nonumber
\end{align}
Then $\nu_{1,\ep_2}=\tilde\nu_{1,\ep_2}/|\tilde\nu_{1,\ep_2}|$ is 
a unit normal vector field for $b$.
We can see $b_{1,\ep_2}:{\cal D}_{\ep_2}\to\R^3$ are fronts
by a similar method as in Section \ref{sec:d4mdisc}.
The singular set $S(b)$ are
\begin{equation}\label{eq:impld4660}
S(b)=
\left\{
\begin{array}{ll}
\{(\theta,z)\,|\,\sinh3\theta=0\}=
\{(0,z)\}\ & (\ep_2=1),\\
\{(\theta,z)\,|\,\cosh3\theta=0\}=\{z=0\}\ &(\ep_2=-1).
\end{array}
\right.
\end{equation}

\section{Geometry of bifurction sets}\label{sec:geombif}
\subsection{Asymptotic behavior of parametrization near singular points}
Let $b=b_{\ep_1,\ep_2}$ be the parametrization of the bifurcation set
as in
Sections \ref{sec:d4mdisc} and \ref{sec:d4pdisc}, and
$\nu=\nu_{\ep_1,\ep_2}$ its unit normal vector field.
We set the coefficients of the first and second fundamental forms
as follows:
\begin{equation*}
\begin{array}{lll}
E=E_{\ep_1,\ep_2}=b_\theta\cdot b_\theta,&
F=F_{\ep_1,\ep_2}=b_\theta\cdot b_z,&
G=G_{\ep_1,\ep_2}=b_z\cdot b_z,\\
L=L_{\ep_1,\ep_2}=-b_\theta\cdot \nu_\theta,&
M=M_{\ep_1,\ep_2}=-b_\theta\cdot \nu_z,&
N=N_{\ep_1,\ep_2}=-b_z\cdot \nu_z.
\end{array}
\end{equation*}
We have the following lemma.
\begin{lemma}\label{eq:asymbeh}
The coefficients of the first and second fundamental forms
satisfy
\begin{align}
E&=z^4\lambda(\theta)^2a(\theta,z)^2
\big(E_0(\theta)+zE_1(\theta)+z^2E(\theta,z)\big),\label{eq:efactor}\\
F&=z^3\lambda(\theta)a(\theta,z)
\big(F_0(\theta)+zF_1(\theta,z)\big),\label{eq:ffactor}\\
G&=1+z^2G_0(\theta,z)+z^3G_1(\theta)+z^4G_2(\theta,z),\label{eq:gfactor}\\
L&=z^2\lambda(\theta)a(\theta,z)|\tilde\nu|^{-1}
\big(L_0+zL_1(\theta)+z^2L_2(\theta)+z^3L_3(\theta)
+z^4L_4(\theta,z)\big),\label{eq:lfactor}\\
M&=z^2\lambda(\theta)a(\theta,z)|\tilde\nu|^{-1}
\big(M_0(\theta)+zM_1(\theta,z)\big),\label{eq:mfactor}\\
N&=|\tilde\nu|^{-1}
\big(N_0(\theta)+zN_1(\theta)+z^2N_2(\theta,z)\big),\label{eq:nfactor}
\end{align}
where\/
$E_i,F_i,G_i,L_i,M_i,G_i$ $(i=0,1,\ldots,4)$ are functions
of variables\/ $\theta$ or\/ $(\theta,z)$ as indicated,
and\/
$a(\theta,z)=a_{\ep_1,\ep_2}(\theta,z)$,
$\lambda(\theta)=\lambda_{\ep_1,\ep_2}(\theta)$,
$\tilde \nu(\theta,z)=\tilde \nu_{\ep_1,\ep_2}(\theta,z)$
are defined in\/ \eqref{eq:impld4170}, \eqref{eq:impld4180},
\eqref{eq:impld4190}, \eqref{eq:bifhyp0100},
\eqref{eq:bifhyp0200}, \eqref{eq:bifhyp0300}
respectively.
Moreover, $E_0(\theta)>0$ and\/
$L_0$ is a non-zero constant.
\end{lemma}
\begin{proof}
By \eqref{eq:impld4150} and \eqref{eq:impld4620},
it holds that the first two components of
$b_\theta$ are $z^2a\lambda AV $, and the third component of $b_\theta$ is 
zero, where
$A=A_{\ep_1,\ep_2}$, $V =V _{\ep_1,\ep_2}$.
Then we see
\eqref{eq:efactor}, \eqref{eq:lfactor}, \eqref{eq:mfactor}
and $E_0>0$.
By $x_z(\theta,0)=y_z(\theta,0)=0$, 
it holds that $b_z(\theta,0)=(0,0,1)$, and this shows
\eqref{eq:gfactor}.
To show \eqref{eq:ffactor},
since the third component of $b_\theta$ is zero, and $b_z(\theta,0)=(0,0,1)$, 
we see \eqref{eq:ffactor}.
We show $L_0$ is a non-zero constant in the case of $\ep_2=1$.
It is sufficient to show
\begin{equation}\label{eq:l0nonzero}
A^{-1}\pmt{-\cos\theta\\ \sin\theta \\ 0}\cdot
(\tilde\nu)_\theta\ne0.
\end{equation}
Since $X_{x\theta}=Y_{x\theta}=X_{y\theta}=Y_{y\theta}=0$ on $z=0$,
the right-hand side of \eqref{eq:l0nonzero}
is
$$
\pmt{-\cos\theta Y_y \\ \cos\theta Y_x+\sin\theta X_x}\cdot
\pmt{-\sin\theta Y_x+\cos\theta X_x \\-\sin\theta Y_y}
=
X_x(\theta,0)Y_y(\theta,0)=g_{1,100}\,g_{2,010}
$$
on $\{z=0\}$.
Thus this is a non-zero constant,
and the other cases can be shown 
by the same calculation.
Finally, \eqref{eq:nfactor} is clear.
\end{proof}

\subsection{Principal curvatures and principal directions}
In this section, we first show Theorem \ref{thm:bddunbdd}.
\begin{proof}[Proof of Theorem\/ {\rm \ref{thm:bddunbdd}}]
Let $K$ and $H$ be the Gaussian curvature and
the mean curvature. By Lemma \ref{eq:asymbeh}, one can see
$z^2\lambda(\theta)K(\theta,z)$ and 
$z^2 \lambda(\theta)H(\theta,z)$ are 
$C^\infty$ functions.
We define $\tilde K,\tilde H,\tilde I$ by
\begin{equation}\label{eq:gaussmeanset}
K=\dfrac{\tilde K}{z^2\lambda
a|\tilde\nu|^2\tilde I},\quad
H=\dfrac{\tilde H}{2z^2\lambda a|\tilde\nu|\tilde I},\quad
\tilde I=\dfrac{EG-F^2}{z^4\lambda^2a^2},
\end{equation}
where we include the variables $(\theta)$ and $(\theta,z)$
in the notation for functions, when it is clear.
Here, 
\begin{equation}\label{eq:tildehki}
\begin{array}{rl}
\tilde H=
&L_0+L_1z+
(G_0L_0 + L_2+\lambda aE_0N_0)z^2+H_3z^3+z^4O(0),\\
H_3=&
G_0L_1+L_0G_1+L_3+\lambda
(-2aF_0M_0
+a E_1 N_0
+ aE_0N_1
+ (a)_zE_0N_0),\\
\tilde K=&L_0N_0+(L_1N_0+L_0N_1)z+z^2O(0),\\
\tilde I=&E_0+zE_1+z^2O(0),
\end{array}
\end{equation}
where the term $z^iO(0)$ stands for a function of the form
$z^ih(\theta,z)$.
Then
$$
\sqrt{H^2-K}=
\dfrac{1}{2z^2|\lambda a||\tilde\nu|\tilde I}
\sqrt{\tilde H^2-4z^2\lambda a\tilde I\tilde K}.
$$
We note that the function 
$f(x)=\sqrt{b_1(\theta,z)^2-4xb_2(\theta,z)}$ 
of the variable $x$
($b_1\ne0$)
satisfies that $f(0)=|b_1(\theta,z)|$.
Thus there exists a function $\tilde f(\theta,z,x)$ such that
\begin{equation}\label{eq:sqroot}
f(x)=|b_1(\theta,z)|-x \tilde f(\theta,z,x).
\end{equation}
Substituting $x=z^2\lambda(\theta)$ into \eqref{eq:sqroot},
there exists a function
$\tilde h(\theta,z)$ such that
$$
\sqrt{\tilde H^2-4z^2\lambda a\tilde I\tilde K}
=
|\tilde H|-z^2\lambda\tilde h.
$$
Thus
$$
H\pm\sqrt{H^2-K}=
\left\{
\begin{array}{l}
\dfrac{\tilde H}{z^2\lambda a|\tilde\nu|\tilde I}
\pm\sgn\lambda
\dfrac{\tilde h}{2|a||\tilde\nu|\tilde I}\\
\pm\sgn\lambda
\dfrac{\tilde h}{2|a||\tilde\nu|\tilde I}
\end{array}
\right..
$$
We see 
$\tilde H(\theta,0)\ne0$ and $|a||\tilde\nu|\tilde I(\theta,0)\ne0$.
This implies that one of
the principal curvatures of $f$ is unbounded
near $z=0$ and $S(f)$,
and the other is bounded.
This proves the assertion.
\end{proof}
By \eqref{eq:tildehki}, we see that there are no vanishing mean curvature points, no
umbilic points, and no points with zero unbounded principal curvature
near the singular point.
We call the {\it unbounded principal curvature\/} the unbounded
one, and the {\it bounded principal curvature\/} the bounded one.
We give evaluation formulas for these principal curvatures,
such as $\tilde h(\theta,z)$ when $\tilde H(0)>0$.
By
the square root function ($b_0\ne0$) is expanded as
\begin{align*}
&\sqrt{b_0^2+b_1z+b_2z^2+b_3(z)z^3}\\
=&
|b_0|+
z\dfrac{b_1}{2 |b_0|}
+
z^2\dfrac{-b_1^2 + 4 b_0^2 b_2}{8 |b_0|^3}
+
z^3\dfrac{b_1^3 - 4 b_0^2 b_1 b_2 + 8 b_0^4 b_3(0)}{16 |b_0|^5}
+z^4O(0),
\end{align*}
we have
\begin{align}
\sqrt{\tilde H^2-4z^2\lambda a\tilde I\tilde K}=
&
L_0+L_1z
+(G_0L_0+L_2-\lambda aE_0N_0)z^2\label{eq:h2mzk}\\
&+
(
G_0 L_1+G_1L_0+L_3
-2\lambda aF_0M_0
-\lambda aE_0N_1 \nonumber\\
&-a E_1 \lambda N_0
- N_0E_0 \lambda(a)_z
)z^3+z^4O(0).\nonumber
\end{align}
Hence
$$
\tilde h(\theta,z)=2aE_0N_0
+2(aE_0N_1+aE_1N_0+E_0N_0(a)_z)z+z^2O(0).
$$
Thus the unbounded principal curvature $\kappa_1$ and the bounded
principal curvature $\kappa_2$ can be written as
\begin{align}
\kappa_1=&
\dfrac{1}{\lambda a|\tilde\nu|\tilde Iz^2}
(L_0+L_1z+z^2O(0)),\label{eq:kappa1}\\
\kappa_2=&
\dfrac{1}{a|\tilde\nu|\tilde I}
\big(aE_0N_0+
(aE_1N_0+aE_0N_1+(a)_zE_0N_0)z
+z^2O(0)\big).\label{eq:kappa2}
\end{align}
By a direct calculation, 
the principal directions with respect to $\kappa_1$ and $\kappa_2$
are
\begin{equation}\label{eq:asymprindir}
\begin{array}{ll}
V_1&=a\lambda|\tilde\nu|\tilde Iz^2(-N+\kappa_1G,M-\kappa_1F)
=(L_0+zO(0),z^2O(0)),\\
V_2&=\dfrac{|\tilde\nu|}{z^2\lambda a}
(-M+\kappa_2F,L-\kappa_2E)
=(-M_0+zO(0),L_0+zO(0)),
\end{array}
\end{equation}
respectively, under the identification 
$\partial_\theta=(1,0)$ and
$\partial_z=(0,1)$.
Since the zero set of a function $\delta(\theta,z)$ is a regular curve
and is transverse to the $\theta$-axis if $\delta_\theta(\theta,0)\ne0$
at $(\theta,0)$ which satisfies $\delta(\theta,0)=0$,
setting ${}'=\partial/\partial \theta$,
the following proposition holds:
\begin{proposition}\label{lem:ridges}
The ridge curve with respect to the
unbounded principal curvature
emanates from the direction\/ $\theta$ satisfying\/
$
a(\theta,0) (2 \lambda'E_0+3\lambda E_0')+2 \lambda E_0 a'(\theta,0)=0
$ under the condition\/
$
\big(a(\theta,0) (2 E_0\lambda'+3 E_0'\lambda)+2 E_0 \lambda a'(\theta,0)\big)'\ne0
$.
The ridge curve with respect to the
bounded principal curvature
emanates from the direction\/ $\theta$ satisfying\/
$
-E_1 L_0 N_0+E_0'M_0 N_0 +2 E_0 (L_0 N_1-M_0 N_0')=0$ under the condition\/
$
\big(-E_1 L_0 N_0+E_0'M_0 N_0 +2 E_0 (L_0 N_1-M_0 N_0')\big)'\ne0$.
There are no subparabolic curves with respect to the unbounded
principal curvature.
The subparabolic curve with respect to the
bounded principal curvature
emanates from the direction\/ $\theta$ satisfying\/
$2 E_0 N_0'-E_0'N_0 =0$ under the condition\/
$\big(2 E_0 N_0'-E_0'N_0 \big)'\ne0$.
\end{proposition}
\begin{proof}
Using \eqref{eq:kappa1}, \eqref{eq:kappa2} and \eqref{eq:asymprindir},
we have the assertions by direct calculations.
We note that 
$\tilde \nu\cdot\tilde \nu=E_0$,
$(\tilde \nu\cdot\tilde \nu)'=E_0'$,
$(\tilde \nu\cdot\tilde \nu)_z=E_1$
at $(\theta,0)$.
\end{proof}
\subsection{Singular curvature and limiting normal curvature}
As the set of singular points
near the singular point consists of a cuspidal edge,
we calculate the singular curvature and limiting normal curvature.
For simplicity, we just give conditions for whether their limits vanish or not.
By \eqref{eq:kskn}, $\kappa_n$ vanishes if and only if $N_0$ vanishes.
Thus if $\ep_1=-1$, then $\kappa_n=0$ at $(0,0)$ 
(respectively, $(\pi/3,0)$, $(2\pi/3,0)$)
if and only if
$g_{2,002}=0$ (respectively, 
$-\sqrt{3} g_{1,002}-g_{2,002}=0$,
$-\sqrt{3} g_{1,002}+g_{2,002}=0$).
These will appear as conditions for the configurations of parabolic curves
in Section \ref{sec:config}.
Furthermore, if $\ep_1=1$, then $\kappa_n=0$ at $(0,0)$
if and only if
$g_{2,002}=0$.
On the other hand, by \eqref{eq:kskn}, \eqref{eq:impld4630}, \eqref{eq:impld4635}
and $\trans{(x_{zz},y_{zz})}=\det A_{\ep_1,\ep_2}\trans{(X_{zz},X_{zz})}$ at $z=0$,
we obtain that
$\kappa_s=0$ at $(0,0)$ 
(respectively, $(\pi/3,0)$, $(2\pi/3,0)$) if and only if
$K_1=0$ (respectively, $K_2+2\sqrt{3}K_3=0$, $-K_2+2\sqrt{3}K_3=0$),
where
\begin{align*}
K_1=& g_{1,100} g_{2,002} g_{2,100} 
- (g_{2,010}^2 + g_{2,100}^2) (g_{1,002} + 3 g_{3,001})\\
K_2=&
(- 2 g_{1,002}+3 g_{3,001}^2) (g_{2,010}^2 + g_{2,100}^2) 
+ g_{1,100} (2 g_{2,002} g_{2,100} + 9 g_{1,100} g_{3,001})\\
K_3=&
 g_{1,100} (g_{1,100} g_{2,002} - g_{1,002} g_{2,100} + 3 g_{2,100} g_{3,001}^2).
\end{align*}
Furthermore, if $\ep_1=1$, then $\kappa_n=0$ at $(0,0)$
if and only if 
$K_1
=0$.
\subsection{Configurations of parabolic curves emanating from 
$D_4^\pm$-singularities}
\label{sec:config}
In this section, we consider
configurations of parabolic curves emanating from $D_4^\pm$-singularities.
By \eqref{eq:gaussmeanset} and \eqref{eq:tildehki}, 
and the same reason just before Proposition \ref{lem:ridges},
if\/ $N_0'(\theta)\ne0$ 
for any\/ $\theta$ satisfying\/ $N_0(\theta)=0$,
then the parabolic curve emanating out at the singular point
emanates in the direction\/ $\theta$ satisfying\/
$N_0(\theta)=0$ at $(\theta,0)$ in $\hat M_{\ep_2}$ or ${\cal D}_{\ep_2}$.
We set $\xi=g_{2,002}$, $\eta =g_{1,002}$ and $\zeta =g_{3,001}$.
We assume $\xi\ne0$ and $\eta -\zeta ^2\ne0$.
Moreover, we assume 
$N_0'(\theta)\ne0$ for any $\theta$ such that $N_0(\theta)=0$.
By $b_z(\theta,0)=(0,0,1)$, \eqref{eq:impld4160} and \eqref{eq:impld4630},
we see that $N_0(\theta)=0$ is equivalent to
\begin{align}
\xi \cos\theta+\eta  \sin\theta+
\zeta ^2\sin3\theta=0\quad&(\ep_1=-1,\ep_2=1),\label{eq:geomd41500}\\
-\xi \sin\theta+\eta  \cos\theta-
\zeta ^2\cos3\theta=0\quad&(\ep_1=-1,\ep_2=-1),\label{eq:geomd41510}\\
-\xi \cosh\theta+\eta  \sinh\theta+
\zeta ^2\sinh3\theta=0\quad&(\ep_1=1,\ep_2=1),\label{eq:geomd41520}\\
-\xi \sinh\theta+\eta  \cosh\theta-
\zeta ^2\cosh3\theta=0\quad&(\ep_1=1,\ep_2=-1),\label{eq:geomd41530}
\end{align}
by the assumption $\xi\ne0$ and $\eta -\zeta ^2\ne0$, and we see 
$\cos\theta=0$, $\sin\theta=0$ are not the solutions of 
\eqref{eq:geomd41500}, \eqref{eq:geomd41510},
and also $\sinh\theta=0$ is not the
solution of \eqref{eq:geomd41520}, \eqref{eq:geomd41530}.

\subsubsection{The case of $D_4^-$-singularities}
Setting $\ep_2=1$ and $t=\cot\theta$, 
the equation \eqref{eq:geomd41500} is equivalent to
\begin{equation}\label{eq:confd4m0100}
p(t)=
\xi t^3+(\eta +3 \zeta ^2) t^2+\xi t+\eta -\zeta ^2=0.
\end{equation}
Thus the number of asymptotic curves emanating from the singular point
is generically $3$ or $1$.
In this case, the singular set is $t=\pm1/\sqrt{3}$.
We consider where the solutions $t$ of \eqref{eq:confd4m0100}
are in $(-\infty,-1/\sqrt{3})$, $(-1/\sqrt{3},1/\sqrt{3})$ or
$(1/\sqrt{3},\infty)$.
Let $D_{-1}$ be the cubic discriminant of \eqref{eq:confd4m0100}:
$$
D_{-1}/4=-\xi^4 - 2 \xi^2 \eta ^2 - \eta ^4 + 24 \xi^2 \eta  \zeta ^2 
- 8 \eta ^3 \zeta ^2 - 18 \xi^2 \zeta ^4 - 18 \eta ^2 \zeta ^4 + 27 \zeta ^8.
$$
We set the following notation for the direction of the parabolic curves.
We dividing $\hat \M_1$ into
$\Omega_1=\{(\theta,r)\in \hat\M_1\,|\, \cot\theta<-1/\sqrt{3}\}$,
$\Omega_2=\{(\theta,r)\in \hat\M_1\,|\, -1/\sqrt{3}<\cot\theta<1/\sqrt{3}\}$,
$\Omega_3=\{(\theta,r)\in \hat\M_1\,|\, \cot\theta>1/\sqrt{3}\}$.
For the case of $\ep_2=-1$, setting $t=\tan\theta$, 
the equation \eqref{eq:geomd41510} is equivalent to $p(-t)=0$.
Thus the cubic discriminant is the same as $D_{-1}$.
We also divide $\hat \M_{-1}$ into 
$\tilde\Omega_1=\{(\theta,r)\in \hat\M_{-1}\,|\, \tan\theta<-1/\sqrt{3}\}$,
$\tilde\Omega_2
=\{(\theta,r)\in \hat\M_{-1}\,|\, -1/\sqrt{3}<\tan\theta<1/\sqrt{3}\}$,
$\tilde\Omega_3=\{(\theta,r)\in \hat\M_{-1}\,|\, \tan\theta>1/\sqrt{3}\}$.
By \eqref{eq:impld4630} and \eqref{eq:impld4195},
the directions defined by 
$\tilde\nu_{-1,1}(0,z)$ and $\tilde\nu_{-1,-1}(\pi/2,z)$
(respectively, 
$\tilde\nu_{-1,1}(\pi/3,z)$ and $\tilde\nu_{-1,-1}(\pi/2+\pi/3,z)$,
$\tilde\nu_{-1,1}(2\pi/3,z)$ and $\tilde\nu_{-1,-1}(\pi/2+2\pi/3,z)$)
are continuously connected across $\{z=0\}$.
Thus, the region 
$\Omega_1$ corresponds to $\tilde\Omega_3$,
$\Omega_3$ corresponds to $\tilde\Omega_1$,
and
$\Omega_2$ corresponds to $\tilde\Omega_2$.
The notation $(ijk|lmn)$, $(i,j,k,l,m,n\in\{1,2,3\})$ stands for
three parabolic curves of $b_{1,1}$ emanating from the singular point, as
they emanate into the regions $\Omega_i,\Omega_j,\Omega_k$,
and
three parabolic curves of $b_{1,-1}$ emanating from the singular point, as
they emanate into the regions $\tilde\Omega_l,\tilde\Omega_m,\tilde\Omega_n$,
respectively.
We set $c_1=\sqrt{3}\xi+\eta +3\zeta ^2$, $c_2=\sqrt{3}\xi-\eta -3\zeta ^2$.
We have the following theorem:
\begin{theorem}\label{thm:numberpara}
The regions\/ $\Omega_1$, $\Omega_2$, $\Omega_3$, 
$\tilde\Omega_1$, $\tilde\Omega_2$, $\tilde\Omega_3$
that the parabolic curves emanating from the singular point
which emanate into 
are summarized as in the
``configurations'' column of Table\/ {\rm \ref{tab:number}} according to the
sign of\/ $D_{-1}$.
\end{theorem}
\begin{table}[!ht]
\begin{center}
\begin{tabular}{|c|c|c|c|c|c|c|c|}
\hline
Case name&$\xi$&$-\xi+\sqrt{3}\eta $&$\xi+\sqrt{3}\eta $&$c_1$&$c_2$&
\multicolumn{2}{|c|}{configurations}\\
\cline{7-8}
&&&&&&$D_{-1}>0$&$D_{-1}<0$\\
\hline
(I-1)       &$+$&$+$&$+$&$+$&any&$(223|122)$&$(3|1)$\\ \hline
(I-2-1-1)   &$+$&$-$&$+$&$+$&$+$&$(222|222)$&$(2|2)$\\ \hline
(I-2-1-2)   &$+$&$-$&$+$&$+$&$-$&$(233|112)$&$(2|2)$\\ \hline
(I-2-2-1)   &$+$&$-$&$-$&any&$+$&$(122|223)$&$(1|3)$\\ \hline
(I-2-2-2)   &$+$&$-$&$-$&$+$&$-$&$(133|113)$&$(1|3)$\\ \hline
(II-1-1-1)  &$-$&$+$&$+$&any&$-$&$(122|223)$&$(1|3)$\\ \hline
(II-1-2-1)&$-$&$+$&$-$&$+$&$-$&$(112|233)$&$(2|2)$\\ \hline
(II-1-2-2)&$-$&$+$&$-$&$-$&$-$&$(222|222)$&$(2|2)$\\ \hline
(II-2-2-1)  &$-$&$-$&$-$&$+$&$-$&$(113|133)$&$(3|1)$\\ \hline
(II-2-2-2)  &$-$&$-$&$-$&$-$&any&$(223|122)$&$(3|1)$\\ \hline
\end{tabular}
\caption{Configurations of parabolic curves.}
\label{tab:number}
\end{center}
\end{table}
Needless to say, geometrically, all of the cases
$(iij)$ $(i\ne j)$ are the same.
Thus we can draw the pictures of the configurations of parabolic curves
in Figure \ref{fig:param}.
We draw ``open'' pictures instead of the usual pictures 
as in Figure \ref{fig:usualopen},
with the parabolic curves drawn as thick lines,
the set of singular points drawn thin as lines,
and 
the intersection curves drawn as thin dotted lines.
\begin{figure}[ht]
\centering
\begin{tabular}{ccc}
\includegraphics[width=.18\linewidth]{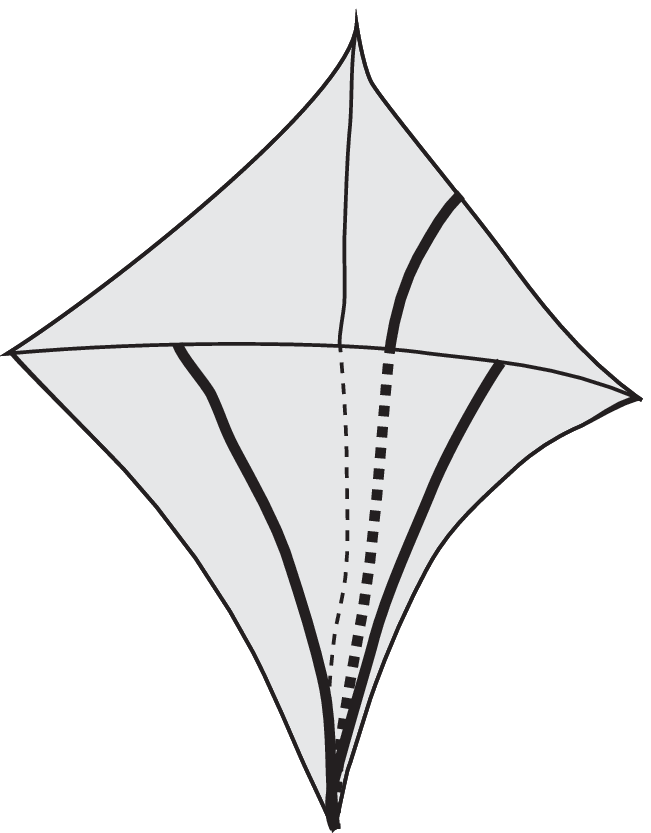}
&\hspace{1mm}&
\includegraphics[width=.2\linewidth]{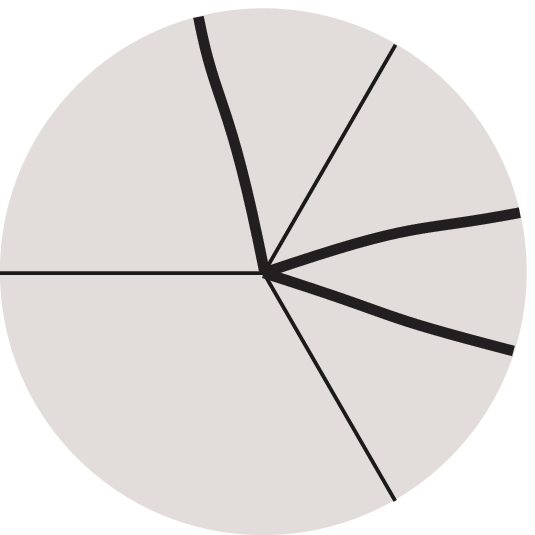}\\
\includegraphics[width=.17\linewidth]{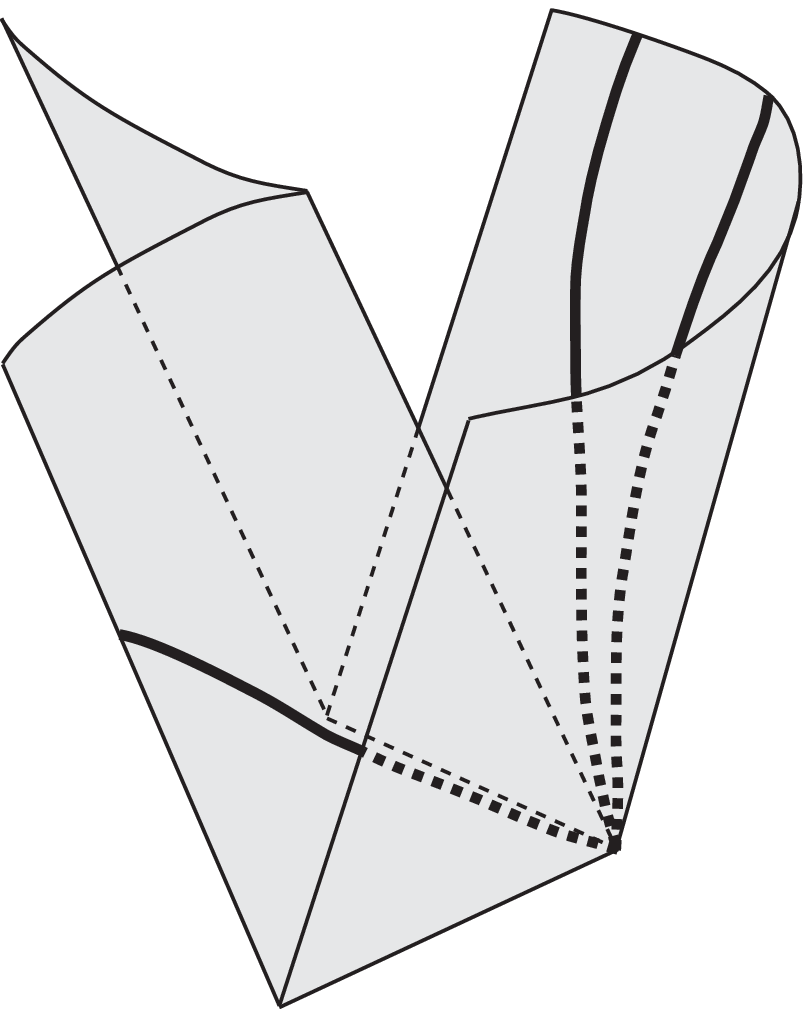}
&\hspace{1mm}&
\includegraphics[width=.2\linewidth]{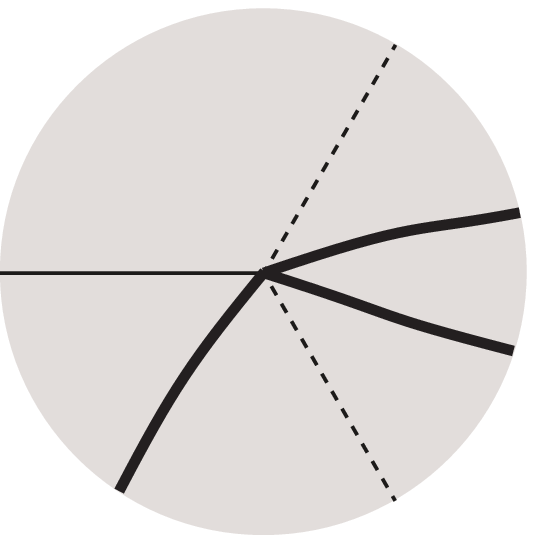}\\
\end{tabular}
\caption{Usual pictures (left) and ``open'' pictures (right)}
\label{fig:usualopen}
\end{figure}
\begin{figure}[ht]
\centering
\includegraphics[width=.18\linewidth]{figm113.eps}
\hspace{1mm}
\includegraphics[width=.18\linewidth]{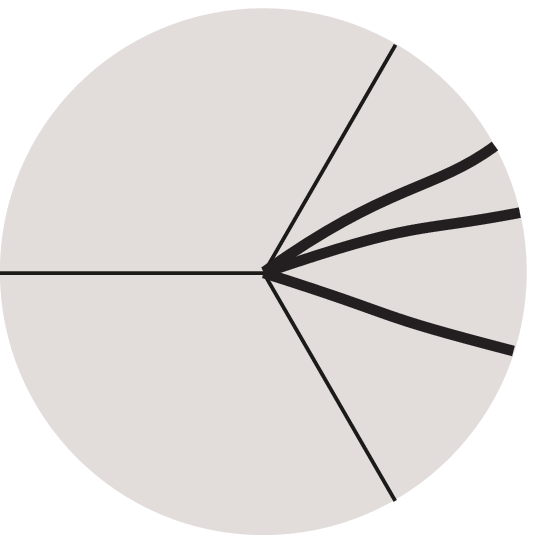}
\hspace{1mm}
\includegraphics[width=.18\linewidth]{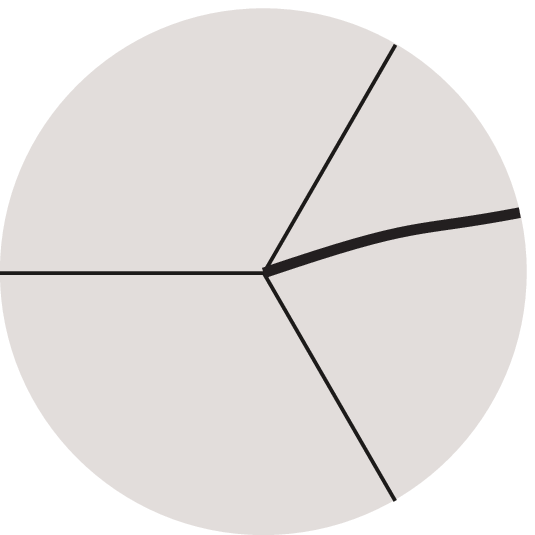}
\caption{The configurations $(112),(222),(1)$ (left to right)
of parabolic curves ($\ep_1=-1$)}
\label{fig:param}
\end{figure}
To show the theorem, we use the following fact, known as Budan's theorem or 
Descartes' rule of signs
(see \cite{descarte}, for example):
\begin{fact}\label{fact:budan}
Let\/ $p(t)$ be a polynomial in\/ $t$.
Then the number of roots of\/ $p(t)$ that are greater than\/ $\alpha$ 
is the same or less than 
the number of sign changes in the sequence of the coefficients
of\/ $p(t+\alpha)$, and their difference is an even number.
Furthermore, the number of roots of\/ $p(t)$ that are less than\/ $\alpha$ 
is the same or less than 
the number of sign changes in the sequence of the coefficients
of\/ $p(-t+\alpha)$, and their difference is an even number.
\end{fact}
\begin{proof}[Proof of Theorem\/ {\rm \ref{thm:numberpara}}]
We assume $\ep_2=1$.
Using Fact \ref{fact:budan},
to study the numbers of roots of $p(t)=0$ in each interval
$(-\infty,-1/\sqrt{3})$, $(-1/\sqrt{3},1/\sqrt{3})$, 
$(1/\sqrt{3},\infty)$,
we look at the numbers of sign changes of
\begin{align*}
p_1(t)=
9p( t+1/\sqrt{3})=& 9\xi t^3+9c_1t^2+6\sqrt{3} c_1t+4\sqrt{3}( \xi+\sqrt{3}\eta ),\\
p_2(t)=
9p( t-1/\sqrt{3})=& 9\xi t^3-9c_2t^2+6\sqrt{3} c_2t+4\sqrt{3}(-\xi+\sqrt{3}\eta ),\\
p_3(t)=
9p(-t+1/\sqrt{3})=&-9\xi t^3+9c_1t^2-6\sqrt{3} c_1t+4\sqrt{3}( \xi+\sqrt{3}\eta ),\\
p_4(t)=
9(-t-1/\sqrt{3})=&-9\xi t^3-9c_2t^2-6\sqrt{3} c_2t+4\sqrt{3}(-\xi+\sqrt{3}\eta ).
\end{align*}
We look at the number of sign changes of these polynomials
under the conditions of each case.
It is summarized in Table \ref{tab:signchange0}.
\begin{table}[ht!]
\centering
\begin{tabular}{|c|c|c|c|c|c|}
\hline
case        &$p_1$&$p_2$&$p_3$&$p_4$\\
\hline
(I-1)
(II-2-2-2)  &   0 &     &     &    1\\
\hline
(I-2-1-1)   &   0 &     &     &    0\\
\hline
(I-2-1-2)   &   0 &    1&     &     \\
\hline
(I-2-2-1)   
(II-1-1-1)  &   1 &     &     &    0\\
\hline
(I-2-2-2)   &   1 &    1&     &     \\
\hline
(II-1-2-1)&     &     &    1&    0\\
\hline
(II-2-2-1)  &     &     &    1&    1\\
\hline
\end{tabular}
\caption{Necessary number of sign changes.}
\label{tab:signchange0}
\end{table}

We first consider the case (I): $\xi>0$.
We assume (I-1): $-\xi+\sqrt{3}\eta >0$.
Then $\eta >0$, $\xi+\sqrt{3}\eta >0$ and $c_1>0$ hold.
Since the number of sign changes in the sequence 
$\xi,c_1,c_1,\xi+\sqrt{3}\eta$
of the coefficients of $p_1$ 
(the number of sign changes in $p_1$) is zero,
there are no roots in $(1/\sqrt{3},\infty)$,
since the number of sign changes in 
$p_4$ i.e., the number of sign changes in 
$-\xi,-c_2,-c_2,-\xi+\sqrt{3}\eta$
is one for any of the cases of $c_2$,
thus the configuration of the roots is $(223)$ or $(3)$ according to the
sign of $D_{-1}$.
We assume (I-2): $-\xi+\sqrt{3}\eta <0$.
We consider the case (I-2-1): $\xi+\sqrt{3}\eta >0$.
Then $\sqrt{3}c_1=2\xi+\xi+\sqrt{3}\eta +3\sqrt{3}\zeta ^2>0$.
Since the number of sign changes in $p_1$ is zero,
there are no roots in $(1/\sqrt{3},\infty)$.
If $c_2>0$, then the number of sign changes in
$p_4$ is 0.
So, in this case we have the configuration is $(222)$ or $(2)$.
If $c_2<0$, then the number of sign changes in
$p_2$ is $1$.
So, in this case we have $(233)$ or $(2)$.
We consider the case (I-2-2): $\xi+\sqrt{3}\eta <0$.
Then the number of sign changes in $p_1$ is $1$.
If $c_2>0$, then the number of sign changes in $p_4$ is $0$.
So, in this case we have $(122)$ or $(1)$.
If $c_2<0$, then the number of sign changes in $p_2$ is $1$.
So, in this case we have $(133)$ or $(1)$.
We second consider the case (II): $\xi<0$.
We assume (II-1): $-\xi+\sqrt{3}\eta >0$.
We consider the case (II-1-1): $\xi+\sqrt{3}\eta >0$.
Then $\eta >0$ and $c_2<0$.
The number of sign changes in $p_1$ is $1$ for any  $c_1$,
and there are no sign changes in $p_4$.
So, in this case we have $(122)$ or $(1)$.
We consider the case (II-1-2): $\xi+\sqrt{3}\eta <0$.
Then $c_2<0$ holds.
If $c_1>0$ (II-1-2-1), 
then the number of sign changes in $p_3$ is $1$, and that in $p_4$
is zero.
So, in this case we have $(112)$ or $(2)$.
If $c_1<0$ (II-1-2-2), then 
the number of sign changes in $p_1$ and $p_3$ are both $0$.
So, in this case we have $(222)$ or $(2)$.
We assume (II-2): $-\xi+\sqrt{3}\eta <0$.
Then $\eta >0$ and $\xi+\sqrt{3}\eta <0$.
If $c_1>0$, then the number of sign changes in $p_3$ and
$p_4$ are both $1$ for any $c_2$.
So this case $(113)$,  $(3)$.
If $c_1<0$, then there is no sign changes in $p_1$ and
the number of sign changes in $p_4$ is $1$.
So this case $(223)$,  $(3)$.
Since the above takes all the possibilities of signs of
$\xi$, $-\xi+\sqrt{3}\eta $, $\xi+\sqrt{3}\eta $, $c_1$ and $c_2$,
we see the assertion.
For the case of $b_{-1,-1}$, since the equation
which we have to consider is $p(-t)$,
the configurations of this case can be obtained by interchanging
$\Omega_1$ with $\tilde \Omega_3$,
$\Omega_2$ with $\tilde \Omega_2$ and
$\Omega_3$ with $\tilde \Omega_1$.
\end{proof}

\subsubsection{The case of $D_4^+$-singularity}
Setting $t=\coth\theta$, 
the equation \eqref{eq:geomd41520} is equivalent to
\begin{equation}\label{eq:confd4m0200}
q(t)=-\xi t^3+(\eta+3 \zeta^2)t^2+\xi t-\eta+\zeta^2=0,
\end{equation}
and setting $t=\tanh\theta$, the equation \eqref{eq:geomd41530}
is equivalent to $q(t)=0$.
Thus the number of asymptotic curves emanating from the singular point
in $({\cal D}_1\cup{\cal D}_{-1})\cap\{r>0\}$
is generically $3$ or $1$ according to the 
cubic discriminant $D_1$ of $q(t)=0$:
$$
D_1/4=\xi^4 - 2 \xi^2 \eta^2 + \eta^4 + 24 \xi^2 \eta \zeta^2 
+ 8 \eta^3 \zeta^2 - 18 \xi^2 \zeta^4 + 
  18 \eta^2 \zeta^4 - 27 \zeta^8.
$$
We divide ${\cal D}_1$ and ${\cal D}_{-1}$ into the regions 
$\Omega_1=\{(\theta,r)\in {\cal D}_{1}\,|\, r>0,\ \coth\theta<-1\}$,
$\Omega_2=\{(\theta,r)\in {\cal D}_{-1}\,|\, r>0\}$,
$\Omega_3=\{(\theta,r)\in {\cal D}_{1}\,|\, r>0,\ \coth\theta>1\}$,
$\tilde\Omega_1=\{(\theta,r)\in {\cal D}_{1}\,|\, r<0,\ \coth\theta<-1\}$,
$\tilde\Omega_2=\{(\theta,r)\in {\cal D}_{-1}\,|\, r<0\}$,
$\tilde\Omega_3=\{(\theta,r)\in {\cal D}_{1}\,|\, r<0,\ \coth\theta>1\}$.
The notation $(ijk|lmn)$, $(i,j,k,l,m,n\in\{1,2,3\})$ stands for 
three parabolic curves of 
$b_{1,1}|_{{\cal D}_1\cap\{r>0\}}$ (if $i,j,k$ is 1 or 3) or 
$b_{1,-1}|_{{\cal D}_{-1}\cap\{r>0\}}$ (if $i,j,k$ is 2) 
emanating from the singular point, as
they emanate into the regions $\Omega_i,\Omega_j,\Omega_k$,
and 
three parabolic curves of 
$b_{1,1}|_{{\cal D}_1\cap\{r<0\}}$ (if $l,m,n$ is 1 or 3) or 
$b_{1,-1}|_{{\cal D}_1\cap\{r<0\}}$ (if $l,m,n$ is 2) 
emanating from the singular point, as
they emanate into the regions $\tilde\Omega_l,\tilde\Omega_m,
\tilde\Omega_n$.
We set $c_3=\eta +3\zeta ^2$, $c_4=-\eta +\zeta ^2$.
We have the following theorem:
\begin{theorem}\label{thm:numberpara2}
The regions\/ $\Omega_1$, $\Omega_2$, $\Omega_3$, 
$\tilde\Omega_1$, $\tilde\Omega_2$, $\tilde\Omega_3$
that the way how the parabolic curves emanate from the singular point
to the regions
are summarized as in the
``configurations'' columns in the Tables\/ 
{\rm \ref{tab:number10}},
{\rm \ref{tab:number20}},
{\rm \ref{tab:number30}} and\/
{\rm \ref{tab:number40}} 
according to the
sign of\/ $D_1$.
\begin{table}[h!]
\centering
\begin{tabular}{|c|c|c|c|c|c|c|c|}
\hline
assumption I:&$3\xi+c_3$&$3\xi-c_3$&$\xi+c_3$&$\xi-c_3$&$c_4$&
\multicolumn{2}{|c|}{configurations}\\
\cline{7-8}
$\xi>0,c_3>0$&&&&&&$D_1>0$&$D_1<0$\\
\hline
(I-1-1)&$+$&$+$&$+$&any&$+$&$(122|223)$&$(1|3)$\\
\hline
(I-1-2)&$+$&$+$&$+$&any&$-$&$(122|223)$&no\\
\hline
(I-2-1)&$+$&$-$&$+$&$-$&$+$&$(122|223)$&$(1|3)$\\
\hline
(I-2-2)&$+$&$-$&$+$&$-$&$-$&$(122|223)$&no\\
\hline
\end{tabular}
\caption{Configurations of parabolic curves when $\xi>0,c_3>0$.}
\label{tab:number10}
\end{table}
\begin{table}[h!]
\centering
\begin{tabular}{|c|c|c|c|c|c|c|c|}
\hline
assumption II:&$3\xi+c_3$&$3\xi-c_3$&$\xi+c_3$&$\xi-c_3$&$c_4$&
\multicolumn{2}{|c|}{configurations}\\
\cline{7-8}
$\xi<0,c_3>0$&&&&&&$D_1>0$&$D_1<0$\\
\hline
(II-1-1)&$+$&$-$&$+$&$-$&$+$&$(223|122)$&$(3|1)$\\
\hline
(II-2-1)&$-$&$-$&$+$&$-$&$+$&$(223|122)$&$(3|1)$\\
\hline
(II-3-1)&$-$&$-$&$-$&$-$&$+$&$(223|122)$&$(3|1)$\\
\hline
(II-1-2)&$+$&$-$&$+$&$-$&$-$&$(223|122)$&no\\
\hline
(II-2-2)&$-$&$-$&any&$-$&$-$&$(223|122)$&$(3|1)$\\
\hline
\end{tabular}
\caption{Configurations of parabolic curves when $\xi<0,c_3>0$.}
\label{tab:number20}
\end{table}
\begin{table}[h!]
\centering
\begin{tabular}{|c|c|c|c|c|c|c|c|}
\hline
assumption III:&$3\xi+c_3$&$3\xi-c_3$&$\xi+c_3$&$\xi-c_3$&$c_4$&
\multicolumn{2}{|c|}{configurations}\\
\cline{7-8}
$\xi<0,c_3<0$&&&&&&$D_1>0$&$D_1<0$\\
\hline
(III-1-1)&$-$&$-$&$-$&$-$&$+$&$(223|122)$&$(3|1)$\\
\hline
(III-2-1)&$-$&any&$-$&$+$&$+$&$(113|133)$&$(3|1)$\\
\hline
\end{tabular}
\caption{Configurations of parabolic curves when $\xi<0,c_3<0$.}
\label{tab:number30}
\end{table}
\begin{table}[h!]
\centering
\begin{tabular}{|c|c|c|c|c|c|c|c|}
\hline
assumption IV:&$3\xi+c_3$&$3\xi-c_3$&$\xi+c_3$&$\xi-c_3$&$c_4$&
\multicolumn{2}{|c|}{configurations}\\
\cline{7-8}
$\xi>0,c_3<0$&&&&&&$D_1>0$&$D_1<0$\\
\hline
(IV-1-1)&any&$+$&$-$&$+$&$+$&$(133|133)$&$(1|3)$\\
\hline
(IV-2-1)&$+$&$+$&$+$&$+$&$+$&$(122|223)$&$(1|3)$\\
\hline
\end{tabular}
\caption{Configurations of parabolic curves when $\xi>0,c_3<0$.}
\label{tab:number40}
\end{table}
\end{theorem}
The pictures of the configurations of parabolic curves
are in Figure \ref{fig:parap}.
\begin{figure}[ht]
\centering
\includegraphics[width=.18\linewidth]{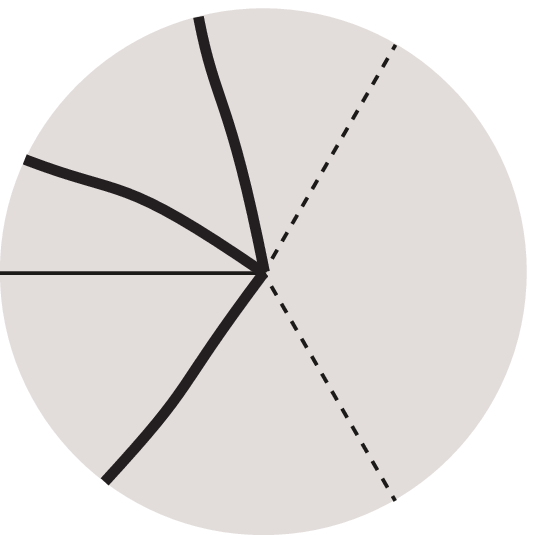}
\hspace{1mm}
\includegraphics[width=.18\linewidth]{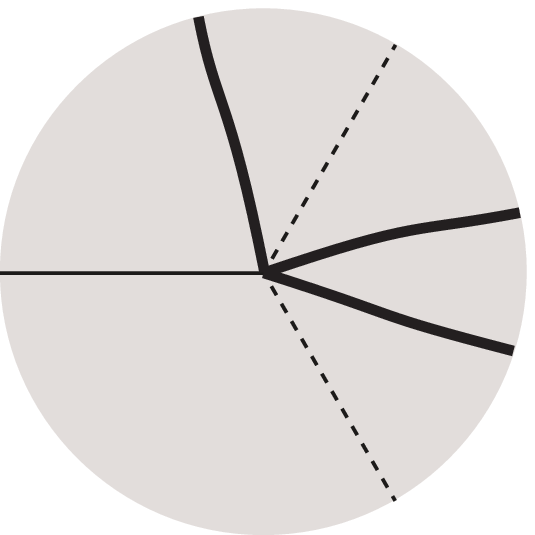}
\hspace{1mm}
\includegraphics[width=.18\linewidth]{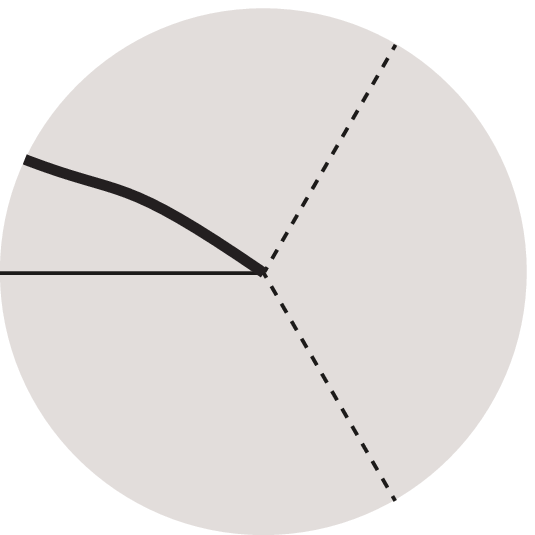}
\hspace{1mm}
\includegraphics[width=.18\linewidth]{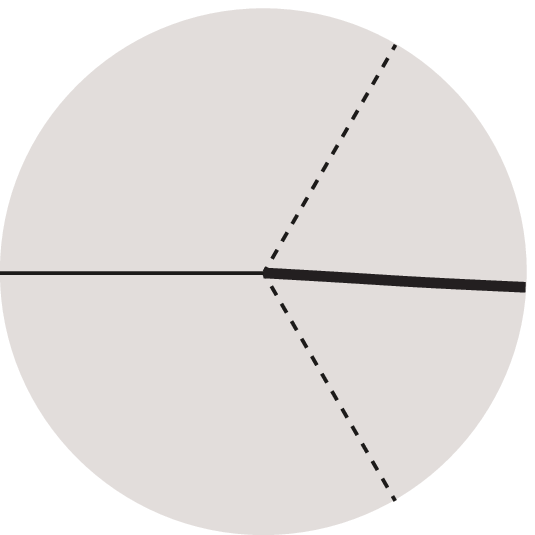}
\caption{The configurations $(113),(122),(1),(2)$ (left to right)
of parabolic curves in the case of $\ep_1=-1$}
\label{fig:parap}
\end{figure}
\begin{proof}
We show the assertion by the same method as in the proof of 
Theorem \ref{thm:numberpara}.
Since the method is completely the same, we state here
just the information about sign changes.
We set 
$q_1(t)=q(t)$ as in \eqref{eq:confd4m0200},
$q_2(t)=q_1(-t)$,
$q_3(t)=q_1(t+1)$,
$q_4(t)=q_1(t-1)$,
$q_5(t)=q_1(-t+1)$ and
$q_6(t)=q_1(-t-1)$.
Then each coefficient of these polynomials is
one of $\xi,c_3,3\xi+c_3,3\xi-c_3,\xi+c_3,\xi-c_3$ and $c_4$.
The necessary number of sign changes of these polynomials
under the conditions of each case are given in Table \ref{tab:signchange},
and
it takes all the possibilities of signs of
$\xi$, $c_3$, $3\xi+c_3$, $3\xi-c_3$, $\xi+c_3$, $\xi-c_3$ and $c_4$,
the assertion for the case of $r>0$ is proven.
\begin{table}[ht!]
\centering
\begin{tabular}{|c|c|c|c|c|c|c|c|}
\hline
case   &$q_1$&$q_2$&$q_3$&$q_4$&$q_5$&$q_6$\\
\hline
(I-1-1),
(I-2-1),
(IV-2-1)&1    &     &1    &     &     &0    \\
\hline
(I-1-2),
(I-2-2)&     &1    &1    &     &     &0    \\
\hline
(II-1-1),
(II-2-1),
(II-3-1),
(III-1-1)&     &1    &0    &     &     &1    \\
\hline
(II-1-2),
(II-2-2)&1    &     &0    &     &     &1    \\
\hline
(III-2-1)&     &1    &     &     &1    &1    \\
\hline
(IV-1-1)&1    &     &1    &1    &     &     \\
\hline
\end{tabular}
\caption{Necessary number of sign changes.}
\label{tab:signchange}
\end{table}
Since $(-r\cosh\theta,-r\sinh\theta)$ is a $\pi$-rotation
of $(r\cosh\theta,r\sinh\theta)$,
the configurations of the case $r<0$ can be obtained by interchanging
$\Omega_1$ with $\tilde \Omega_3$,
$\Omega_2$ with $\tilde \Omega_2$ and
$\Omega_3$ with $\tilde \Omega_1$.
\end{proof}

\subsection{Number of ridges and subparabolic curves}\label{sec:numproof}
Here we show Theorem \ref{lem:ridgenum}.
A pair of two hyperbolic-trigonometric polynomials
$$
g_n(s)=\sum_{1\leq i+j\leq n}a_{i,j}\cosh^i s\,\sinh^j s
\text{\ and\/\ }
h_n(s)=\sum_{1\leq i+j\leq n}b_{i,j}\cosh^j s\,\sinh^i s,
$$
satisfying\/ $a_{i,j}=b_{i,j}=0$ for\/ any odd\/ $i+j$,
or\/ $a_{i,j}=b_{i,j}=0$ for any even\/ $i+j$,
are said to be {\it adapted}\/ if
by setting $\cosh s=1/(1-t^2)^{1/2}$,
$\sinh s=t/(1-t^2)^{1/2}$, ($t=\tanh s$),
the polynomial $(1-t^2)^{n/2}g_n(s)$ with respect to $t$,
and
the polynomial $(1-t^2)^{n/2}h_n(s)/t^n$ with respect to $1/t$,
are the same.
For example, since
$\cosh 3s=(1-t^2)^{-3/2}(1+3t^2)$,
$\sinh 3s=(1-t^2)^{-3/2}t^3(1+3t^{-2})$, the pair
$\cosh 3s$ and $\sinh 3s$ is adapted.
\begin{lemma}\label{lem:trinumber}
{\rm (1)} The number of roots of\/
$f_n(s)=\displaystyle\sum_{i+j=1}^{n}a_{i,j}\cos^i s\,\sin^j s$
is at most\/ $n$ for\/ $s\in [0,\pi)$
if\/ $d(f_n)/ds(s)\ne0$ for all\/ $s$ satisfying\/ $f_n(s)=0$.
{\rm (2)}
We assume that the pair of two hyperbolic-trigonometric polynomials
$g_n(s)$ and $h_n(s)$
is adapted.
Then
the sum of the numbers of roots\/ $s\in\R$ 
of\/ $g_n(s)=0$ and\/ $h_n(s)=0$ 
is at most\/ $n$
under the condition\/
$d(g_n)/ds(s)\ne0$ {\rm (}respectively, $d(h_n)/ds(s)\ne0)$ 
for all\/ $s$ satisfying\/ $g_n(s)=0$ {\rm (}respectively, $h_n(s)=0)$.
\end{lemma}
\begin{proof}
If $s=\pm\pi/2$ are solutions, factoring out
$\cos s$ from $f_n(s)$, and we may assume 
$s=\pm\pi/2$ are not solutions of $f_n(s)=0$.
Setting $\cos s=\pm1/(1+\tan^2s)^{1/2}$ and
$\sin s=\pm\tan s/(1+\tan^2s)^{1/2}$, twice the number of roots of the
equation $f_n(s)=0$
can be reduced to
a polynomial equation with respect to $\tan s$ with degree $2n$.
This shows (1).
See \cite[Lemma 2]{haseimp} for a detailed proof.
For a proof of (2), 
setting $\cosh s=1/(1-t)^{1/2}$ and
$\sinh s=t/(1-t)^{1/2}$, $(t=\tanh s)$,
since $g_n$ and $h_n$ consist only of the terms where $i+j$ is odd or even,
and the equations $g_n(s)=0$, $h_n(s)=0$ 
are reduced to the same polynomial equations with respect to 
$t$ and $1/t$, respectively,
with degree $n$.
Since $\tanh s$ takes value in $(-1,1)$,
we see the assertion.
\end{proof}
The number $n$ of the above $f_n$
(respectively, $g_n$, $h_n$),
is called the {\it degree\/},
and it is denoted by $\deg(f_n)$ (respectively, $\deg(g_n)$, $\deg(h_n)$).
We give a proof of Theorem \ref{lem:ridgenum}
under the same assumption as in Proposition \ref{lem:ridges}.

\begin{proof}[Proof of Theorem \ref{lem:ridgenum}]
We see the degrees of
\begin{align*}
C_1(\theta)=&
a(\theta,0) (2 \lambda'E_0+3\lambda E_0')+2 \lambda E_0 a_z(\theta,0),\\
C_2(\theta)=&
-E_1 L_0 N_0+E_0'M_0 N_0 +2 E_0 (L_0 N_1-M_0 N_0'),\\
C_3(\theta)=&
2 E_0 N_0'-E_0'N_0.
\end{align*}
These appear as the conditions of
Proposition \ref{lem:ridges}.
We easily see 
$\deg(L_0)=\deg(a)=0$,
$\deg(\lambda)=\deg(\lambda')=3$,
and $\lambda$, $\lambda'$ have only odd degree terms.
Furthermore, we see 
$\deg(E_0)=\deg(E_0')=2,\deg(E_1)=4$
and these have only even degree terms.
For the degree of $a_z$, by
\eqref{eq:impld4190}
and
\eqref{eq:bifhyp0300},
 the degrres of $a_z$ is the same as that of $\alpha$ and $\beta$.
Thus $\deg(a_z)=4$,
and it has only even degree terms.
We now consider the degrees of $M_0,N_0,N_0',N_1$.
We assume $\ep_1=\ep_2=1$.
Since $M_0=\inner{\tilde A V }{\tilde \nu_z}(\theta,0)$,
and 
\begin{align*}
\tilde\nu_z&=(-\cosh\theta dY+\sinh\theta dX)_z\\
&=
-\cosh\theta (dY_xx_z+dY_yy_z+dY_z)
+
\sinh\theta (dX_xx_z+dX_yy_z+dX_z)
\end{align*}
at $(\theta,0)$ which has degree $3$, together with
$\deg(V )=1$, $\deg(dY_z)=\deg(dX_z)=2$,
we have $\deg(M_0)=4$.
Moreover, we see $M_0$ has only even degree terms.
Similarly, we see $\deg(N_0)=\deg(N_0')=3$ and they have only odd degree terms.
Since $\deg(\alpha)=\deg(\beta)=4$, it holds that
$\deg(x_{zz})=\deg(y_{zz})=\deg(X_{zz})=\deg(Y_{zz})=4$
 and they have only even degree terms.
Furthermore, since $b_z(\theta,0)=(0,0,1)$, $b_{zz}=(x_{zz},y_{zz},0)$,
$$
\tilde\nu_{zz}=
-\cosh\theta (dY_xx_{zz}+dY_yy_{zz}+dY_{zz})
+
\sinh\theta (dX_xx_{zz}+dX_yy_{zz}+dX_{zz})
$$
at $(\theta,0)$ and
$2(\alpha\sinh\theta-\beta\cosh\theta)=\sinh3\theta$,
we have $\deg(N_1)=7$, and has only odd degree terms.
Thus $\deg(C_1)=\deg(C_2)=9$ and $\deg(C_3)=5$.
On the other hand,
by \eqref{eq:impld4540}, \eqref{eq:bifhyp0100} and \eqref{eq:bifhyp0200},
the degrees of 
$C_i$ $(i=1,2,3)$ in the case of $\ep_2=-1$
are the same as the case of $\ep_2=1$.
Moreover, for each $i=1,2,3$, the pair
$C_i$ $(\ep_2=\pm1)$
of hyperbolic-trigonometric polynomial
is adapted,
so we have the assertion.
For the case of $\ep_1=-1$, 
we can obtain the degrees by similar calculations.
Summarizing up these degrees, Proposition \ref{lem:ridges}
and the fact that each $D_4^\pm$ singularity consists of two sheets
which correspond to $\ep_2=\pm1$ respectively,
we have the assertion.
\end{proof}


\medskip
\toukouchange{
{\footnotesize
\begin{flushright}
\begin{tabular}{l}
(Saji)\\
Department of Mathematics,\\
Graduate School of Science, \\
Kobe University, \\
Rokkodai 1-1, Nada, Kobe \\
657-8501, Japan\\
  E-mail: {\tt saji@math.kobe-u.ac.jp}\\
\\
(Santos)\\
Departamento de Matem{\'a}tica, Ibilce, \\
Universidade Estadual Paulista (Unesp), \\
R.~Crist{\'o}v{\~a}o Colombo, 2265, Jd Nazareth, \\
15054-000 S{\~a}o Jos{\'e} do Rio Preto, SP, Brazil\\
  E-mail: {\tt samuelp.santos@hotmail.com}
\end{tabular}
\end{flushright}}

\begin{thebibliography}{9}
\bibitem{descarte}
B. Anderson, J. Jackson and M. Sitharam
{\it Descartes' rule of signs revisited}, 
Amer. Math. Monthly {\bf 105} (1998), 
447--451.

\bibitem{AGV} V. I. Arnold, S. M. Gusein-Zade, and A. N. Varchenko,
{\it Singularities of differentiable maps}, Vol. {\bf 1},
Monographs in Mathematics {\bf 82}, Birkh\"auser, Boston, 1985.

\bibitem{fukuihase}
T. Fukui and M. Hasegawa,
{\it Fronts of Whitney umbrella --- a differential 
geometric approach via blowing up}, J. Singul. {\bf 4} (2012), 35--67.

\bibitem{haseimp}
M. Hasegawa,
{\it Parabolic, ridge and sub-parabolic curves on implicit surfaces with singularities},
Osaka J. Math. {\bf 54} (2017), no. 4, 707--721.


\bibitem{hironaka}
H. Hironaka,
{\it Stratification and flatness}. 
Real and complex singularities 
(Proc. Ninth Nordic Summer School/NAVF Sympos. Math., Oslo, 1976), 
pp. 199--265. Sijthoff and Noordhoff, Alphen aan den Rijn, 1977.

\bibitem{izutak}
S. Izumiya and M. Takahashi,
{\it Spacelike parallels and evolutes in Minkowski pseudo-spheres},
J. Geom. Phys. {\bf 57} (2007), no. 8, 1569--1600. 

\bibitem{izugra}
S. Izumiya,
{\it The theory of graph-like Legendrian unfoldings\/{\rm :} 
equivalence relations},
Singularities in generic geometry, 
Adv. Stud. Pure Math., {\bf 78} (2018), 107--161.

\bibitem{irrt-book} 
S. Izumiya, M. C. Romero-Fuster, M. A. S. Ruas and F. Tari,
{\it Differential Geometry from a Singularity Theory Viewpoint}.
World Scientific Pub. Co Inc. 2015.

\bibitem{kruycaus}
M. Kokubu, W. Rossman, M. Umehara and K. Yamada,
{\it Flat fronts in hyperbolic\/ $3$-space and their caustics},
J. Math. Soc. Japan {\bf 59} (2007), no. 1, 265--299.

\bibitem{ms}
L. F. Martins and K. Saji,
{\it Geometric invariants of cuspidal edges},
Canad. J. Math. {\bf 68} (2016), no. 2, 445--462. 

\bibitem{msuy}
L. F. Martins, K. Saji, M. Umehara and K. Yamada,
{\it Behavior of Gaussian curvature and mean curvature near 
non-degenerate singular points on wave fronts},
Geometry and topology of manifolds, 247--281, 
Springer Proc. Math. Stat., {\bf 154}, Springer, 2016.

\bibitem{osetari}
R. Oset Sinha and F. Tari,
{\it On the flat geometry of the cuspidal edge},
Osaka J. Math. {\bf 55} (2018), no. 3, 393--421. 

\bibitem{suyfro}
K. Saji, M. Umehara and K. Yamada,
{\em The geometry of fronts},
Ann. of Math. {\bf 169} (2009), no. 2, 491--529.

\bibitem{teraprin}
K. Teramoto,
{\it Principal curvatures and parallel surfaces of wave fronts},
Adv. Geom. {\bf 19} (2019), no. 4, 541--554.

\bibitem{zakall}
V. M. Zakalyukin,
{\it Lagrangian and Legendre singularities},
Funcational Anal. Appl. {\bf 10} (1976), no. 1, 23--31.


\end{thebibliography}
\end{document}